\documentclass[12pt]{amsart}
\usepackage{amssymb,latexsym,amsmath,enumitem,mathrsfs,geometry}
\usepackage{thmtools}
\usepackage{fullpage}
\usepackage{fancyhdr}
\usepackage{xcolor}
\usepackage{bbm}
\usepackage{stmaryrd}
\usepackage{mathrsfs}
\usepackage{hyperref}
\usepackage{lineno}
\usepackage{diagbox}
\usepackage{array}
\usepackage{tikz}
\usetikzlibrary{arrows}
\usetikzlibrary{positioning, arrows.meta, shapes.geometric}
\usepackage{float}
\usepackage{comment}
\usepackage{mathtools}
\usepackage{amsthm}
\usepackage{cleveref}

\newtheorem{theorem}{Theorem}[section]
\newtheorem*{theorem*}{Theorem}

\newtheorem{lemma}[theorem]{Lemma}
\newtheorem*{remark*}{Remark}

\newtheorem{corollary}[theorem]{Corollary}
\newtheorem{proposition}[theorem]{Proposition}

\newtheorem{definition}[theorem]{Definition}
\newtheorem{remark}[theorem]{Remark}

\newcommand{\eps}{\varepsilon}
\newcommand{\R}{\mathbb{R}}
\newcommand{\C}{\mathbb{C}}
\newcommand{\Q}{\mathbb{Q}}

\newcommand{\E}{\mathbf{E}}
\renewcommand{\P}{\mathbf{P}}
\newcommand{\Var}{\mathbf{Var}}
\newcommand{\Cov}{\mathbf{Cov}}
\newcommand{\N}{\mathbb{N}}
\newcommand{\Z}{\mathbb{Z}}

\begin{document}

\numberwithin{equation}{section}

\title{Products of consecutive integers with unusual anatomy}

\author[Tao]{Terence Tao}
\address{Department of Mathematics, UCLA, 405 Hilgard Ave, Los Angeles CA 90024}
\email{tao@math.ucla.edu}

\thanks{}

\date{\today}

\begin{abstract} Call an interval $\{N+1,\dots,N+H\}$ of consecutive natural numbers \emph{bad} if the product $(N+1) \dots (N+H)$ is divisible by the square of its largest prime factor; \emph{very bad} if this product is powerful, and \emph{type $F_3$} if it has the same squarefree component as a factorial $a!$ with $a<N$.  Such concepts arose in the analysis of the factorial equation $a_1! a_2! a_3! = m^2$ with $a_1<a_2<a_3$.  Answering several questions of Erd\H{o}s and Graham, we obtain asymptotics for the number of integers contained in bad or very bad intervals, and near-asymptotics both for the number of right endpoints of type $F_3$ intervals and for the number of solutions to $a_1! a_2! a_3! = m^2$.
\end{abstract}

\maketitle

\section{Introduction} \label{intro}

Given an interval $\{N+1,\dots,N+H\}$ of natural numbers (with $N \geq 0$ and $H \geq 1$), define their \emph{product} to be the quantity
\begin{equation}\label{nhint}
 (N+1) \dots (N+H) = H! \binom{N+H}{H} = \frac{N! (N+H)!}{(N!)^2}.
 \end{equation}
We will be interested in the \emph{anatomy} of such products, by which we mean the properties of their prime factorization.  Such questions have been studied extensively in the literature, in part due to the obvious connection with multiplicative Diophantine equations involving factorials or binomial coefficients.  We recall two classical results in this area:

\begin{theorem}\label{classical}\
    \begin{itemize}
        \item[(i)] (Sylvester--Schur theorem) If $N \geq H$, then the largest prime factor of \eqref{nhint} exceeds $H$.
        \item[(ii)] (Erd\H{o}s--Selfridge theorem) If $H \geq 2$, then the product \eqref{nhint} is never a perfect power.
    \end{itemize}
\end{theorem}

\begin{proof}  For (i), see \cite{sylvester}, \cite{schur} (as well as \cite{erdos-sylvester}, \cite{moser}, \cite{faulkner}, \cite{hanson} for alternate proofs and refinements).  For (ii), see \cite{selfridge}.
\end{proof}

The following unusual anatomical features of \eqref{nhint} were studied in the work of Erd\H{o}s and Graham \cite{erdos-graham}, \cite[p. 70, 71, 73]{eg}:

\begin{definition}[Bad, very bad, $F_3$ intervals]\label{def-bad}  Let $N, H$ be natural numbers.
\begin{itemize}
    \item[(i)] An interval $\{N+1,\dots,N+H\}$ is \emph{bad} if its product \eqref{nhint} is divisible by the square of its largest prime factor\footnote{We adopt the convention that $1$ is not divisible by the square of its largest prime factor, as it has no such factors.}.  For instance, $\{24,25\}$ is bad because $24 \times 25$ is divisible by the square of the largest prime factor $5$.
    \item[(ii)] An interval $\{N+1,\dots,N+H\}$ is \emph{very bad} if its product \eqref{nhint} is \emph{powerful} (or \emph{squarefull}), that is to say that every prime that divides this product, divides it at least twice.  For instance, $\{8,9\}$ is very bad because $8 \times 9 = 2^3 \times 3^2$ is powerful.
    \item[(iii)] An interval $\{N+1,\dots,N+H\}$ is of \emph{type $F_3$} if there exists $1 \leq a < N$ such that the product \eqref{nhint} has the same squarefree component\footnote{The squarefree component $s(n)$ of a natural number is the smallest natural number of the form $n/m^2$ for some perfect square $m^2$.} as $a!$.  Equivalently, by \eqref{nhint}, there is a solution in the natural numbers to the equation
\begin{equation}\label{f3-eq}
    a_1! a_2! a_3! = m^2; \quad a_1 < a_2 < a_3
\end{equation}
with $(a_2,a_3)=(N,N+H)$. For instance, $\{8,9,10\}$ is of type $F_3$ because $6! 7! 10! = (6!\times 7!)^2$ is a solution to \eqref{f3-eq}; equivalently, $s(8 \times 9 \times 10) = 5 = s(6!)$ (indeed $8 \times 9 \times 10 = 6!$).
\end{itemize}
\end{definition}

\begin{remark} As already noted, type $F_3$ intervals were introduced in \cite{erdos-graham} to analyze the solutions to the multiplicative factorial Diophantine equation \eqref{f3-eq}; the most difficult case in their analysis occurred when the type $F_3$ interval was also bad.  Very bad intervals are a subclass of bad intervals that were connected to several further conjectures of Erd\H{o}s and Erd\H{o}s--Selfridge on powerful numbers \cite{erdos76}, \cite{erdos-consecutive}, \cite{eg}; as we shall see in this paper, the study of very bad intervals also serves a toy model for the study of type $F_3$ intervals.  One can also define type $F_k$ intervals for other values of $k$ than $3$, relating to solutions to the equation $a_1! \dots a_k! = m^2$; we will not consider such concepts here (as we unfortunately were not able to make any progress on understanding these intervals better), but refer the reader to \cite{erdos-graham}.
\end{remark}

Following\footnote{The sets denoted $F_k$ in \cite{erdos-graham}, \cite{eg} are not the same as the sets ${\mathcal F}_k$ used here: if one defines ${\mathcal F}_j$ for general $j$ analogously to ${\mathcal F}_3$ in \Cref{def-sets} below, then the set denoted $F_k$ in those papers is $\bigcup_{j=1}^k {\mathcal F}_j$ in our notation, and the set denoted $D_k$ there is the difference between the sets denoted $F_k$ and $F_{k-1}$.  In \cite[p. 345]{erdos-graham}, elements of ${\mathcal B}^1$ and  $\mathcal{VB}$ are called \emph{bad'} and \emph{bad''}, respectively.} \cite{eg}, we associate various sets of natural numbers to these properties:

\begin{definition}[Bad, very bad, $F_3$ sets]\label{def-sets}\
\begin{itemize}
    \item[(i)] We let ${\mathcal B}$ denote the set of natural numbers $n$ that are contained in at least one bad interval $\{N+1,\dots,N+H\}$.
    \item[(ii)] We let $\mathcal{VB}$ denote the set of natural numbers $n$ that are contained in at least one very bad interval $\{N+1,\dots,N+H\}$.
    \item[(iii)] We let ${\mathcal F}_3$ denote the set of natural numbers $n$ that are of the form $N+H$ for some type $F_3$ interval $\{N+1,\dots,N+H\}$.  Equivalently, there exists a solution to \eqref{f3-eq} with $a_3 = n$.
\end{itemize}
By restricting to the case $H=1$, we can define simpler sets ${\mathcal B}^1$, $\mathcal{VB}^1$, ${\mathcal F}_3^1$ inside each of ${\mathcal B}$, $\mathcal{VB}$, ${\mathcal F}_3$, consisting of natural numbers with unusual anatomy:
\begin{itemize}
   \item[$(i)_1$] ${\mathcal B}$ contains the set ${\mathcal B}^1$ of natural numbers $n$ that are divisible by the square of their largest prime factor.
   \item[$(ii)_1$] $\mathcal{VB}$ contains the set $\mathcal{VB}^1$ of powerful numbers.
   \item[$(iii)_1$] ${\mathcal F}_3$ contains the set ${\mathcal F}_3^1$ of natural numbers $n$ for which $n$ has the same squarefree component $s(n)=s(a!)$ as $a!$ for some $1 \leq a < n-1$.
\end{itemize}
\end{definition}

\begin{figure}
    \centering
\includegraphics[width=0.75\textwidth]{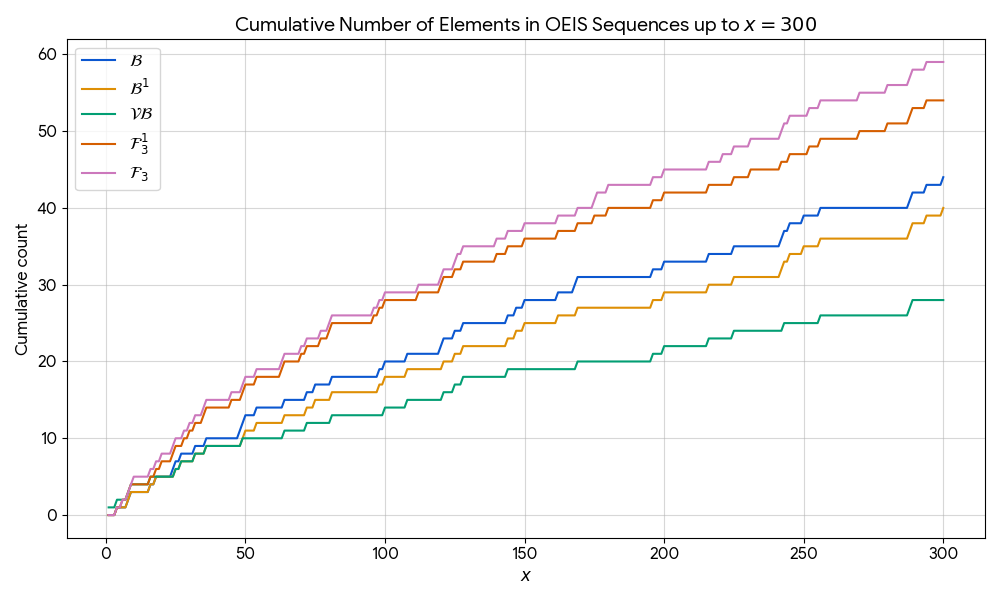}
    \caption{Plots of $\#({\mathcal A} \cap [1,x])$ for ${\mathcal A} = {\mathcal B}, {\mathcal B}^1, \mathcal{VB}, {\mathcal F}_3, {\mathcal F}_3^1$.  The set $\mathcal{VB}^1$ is not depicted as it numerically coincides with $\mathcal{VB}$. (Image generated by Gemini using data from the OEIS.)}
    \label{fig-cumulative}
\end{figure}

It was conjectured in \cite{eg} that these simpler sets have the same asymptotic size as their more complicated containing sets, thus\footnote{See \Cref{notation-sec} for our conventions on asymptotic notation.}
\begin{align}
    \#({\mathcal B} \cap [1,x]) &\sim \#({\mathcal B}^1 \cap [1,x])\label{b-conj}\\
    \#(\mathcal{VB} \cap [1,x]) &\sim \#(\mathcal{VB}^1 \cap [1,x])\label{vb-conj}\\
    \#({\mathcal F}_3 \cap [1,x]) &\sim \#({\mathcal F}_3^1 \cap [1,x])\label{f3-conj}
\end{align}
as $x \to \infty$.  See \Cref{fig-cumulative}, \Cref{fig-bad}, \Cref{fig-vb}, \Cref{fig-f3} for some numerical confirmation of these conjectures.

The purpose of this paper is to establish the first two conjectures (with some quantitative improvements on the error terms), and make partial progress on the third; see \Cref{tab-summary} for a summary.  We now discuss each of them in turn.

\begin{table}
\centering
\renewcommand{\arraystretch}{1.6}
\small
\begin{tabular}{|c|c|c|}
\hline
Set ${\mathcal A}$ & Previous bound on $\#({\mathcal A} \cap [1,x])$ & This paper \\
\hline
${\mathcal B}^1$ & $\frac{x}{z^{2+o(1)}}$ \ \cite{ivic} & --- \\
${\mathcal B} \backslash {\mathcal B}^1$ & $\ll \frac{x}{\exp((c+o(1)) \log^{1/4} x \log_2^{3/4} x)}$ \ \cite{luca} & $\frac{\#({\mathcal B}^1 \cap [1,x])}{\log^{1-o(1)} x}$ \ (\Cref{bad-thm}) \\
\hline
$\mathcal{VB}^1$ & $\sim \frac{\zeta(3/2)}{\zeta(3)} \sqrt{x}$ \ \cite{erdos-szekeres}, \cite{bateman-grosswald} & --- \\
$\mathcal{VB} \backslash \mathcal{VB}^1$ & $\ll \frac{x}{\exp((c+o(1)) \log^{1/4} x \log_2^{3/4} x)}$ \ \cite{luca} & $x^{\frac{2}{5}+o(1)}$ \ (\Cref{verybad-thm}) \\
\hline
${\mathcal F}_3^1$ & $\sim c_3^1 \sqrt{x}$ \ \eqref{f31} & --- \\
${\mathcal F}_3 \backslash {\mathcal F}_3^1$ & $\ll \frac{x}{\exp((c+o(1)) \log^{1/4} x \log_2^{3/4} x)}$ \ \cite{luca} & $x^{\frac{1}{2}+o(1)}$ \ (\Cref{f3-thm}) \\
\hline
\end{tabular}
\caption{Bounds for the six sets studied in this paper.  In particular, the conjectures \eqref{b-conj} and \eqref{vb-conj} now hold, while $\#({\mathcal F}_3 \cap [1,x]) = x^{\frac{1}{2}+o(1)}$, falling just short of \eqref{f3-conj}.  Here $c>0$ is an unspecified absolute constant, and $z$ is defined in \eqref{z-def}.}
\label{tab-summary}
\end{table}

\subsection{Bad intervals}

The first few elements of the set ${\mathcal B}^1$ are
\begin{equation}
4, 8, 9, 16, 18, 25, 27, 32, 36, 49, 50, 54, 64, 72, \dots \tag{${\mathcal B}^1$}
\end{equation}
(\href{https://oeis.org/A070003}{OEIS A070003}; see \cite{oeis}).  The asymptotics of this set are well understood, largely thanks to the very well studied asymptotics of smooth numbers.  For a large number $x$, define the double and triple logarithms
$$ \log_2 x \coloneqq \log\log x, \quad \log_3 x \coloneqq \log\log\log x,$$
the exponent
\begin{equation}\label{u-def}
 u_0 \coloneqq \frac{\sqrt{2} \log^{1/2} x}{\log^{1/2}_2 x},
\end{equation}
and the smoothness threshold
\begin{equation}\label{z-def}
 z \coloneqq \exp\left( \frac{1}{\sqrt{2}} \log^{1/2} x \log^{1/2}_2 x\right) = x^{1/u_0}.
\end{equation}
For future reference we observe the hierarchy of scales
$$ 1 \ll \log_2 z \asymp \log_2 x \ll \log z = \log^{1/2+o(1)} x \ll \log x \ll z \ll x.$$

\begin{remark}
The critical nature of these choices of $u_0$ and $z$ arise from the approximate identity
\begin{equation}\label{upow}
    u_0^{u_0} = z^{1+o(1)}
\end{equation}
(so that the Dickman function $\rho(u_0)$ is of size $z^{-1+o(1)}$).
\end{remark}

\begin{lemma}[Asymptotics for ${\mathcal B}^1$]\label{b-asym}\
    \begin{itemize}
        \item[(i)]      One has
\begin{align*}
\#({\mathcal B}^1 \cap [1,x]) &= \frac{x}{z^{2+o(1)}}\\
&= \frac{x}{\exp\left( (\sqrt{2}+o(1)) \log^{1/2} x \log^{1/2}_2 x\right)}
\end{align*}
as $x \to \infty$.
        \item[(ii)] One has the stability estimate
\begin{equation}\label{bo}
\#({\mathcal B}^1 \cap [1,cx]) \asymp \#({\mathcal B}^1 \cap [1,x])
\end{equation}
whenever $c \asymp 1$ and $x$ is sufficiently large.
    \end{itemize}
\end{lemma}

\begin{proof} For (i), see (the $r=0$ case of) \cite[(1.3), (1.7)]{ivic}, which in fact gives a more precise expansion
\begin{equation}\label{g0}
\#({\mathcal B}^1 \cap [1,x])  = \frac{x}{\exp\left( \left(\sqrt{2}+g_0(x) + O\left(\frac{\log_3^3 x}{\log_2^3 x}\right)\right) \log^{1/2} x \log^{1/2}_2 x\right)}
\end{equation}
where
$$ g_0(x) \coloneqq  \frac{\log_3 x - 2 -\log 2}{2 \log_2 x} \left( 1 + \frac{2}{\log_2 x} \right) - \frac{(\log_3 x - \log 2)^2}{8 \log_2^2 x};$$
see \Cref{fig-bad}. For the reader's convenience we also give a more self-contained proof of part (i) in \Cref{notation-sec}, where we also establish (ii).
In \cite[p. 345]{erdos-graham} the exponent $\sqrt{2}$ was incorrectly reported as $1$ as an ``old result of one of the authors''; but in the later paper \cite[p. 73]{eg} by the same authors, this exponent $1$ was replaced with an unspecified constant $c$. See also \cite{ivic-2} for a discussion of some other related results with incorrectly reported constants.
\end{proof}

\begin{figure}
    \centering
\includegraphics[width=0.75\textwidth]{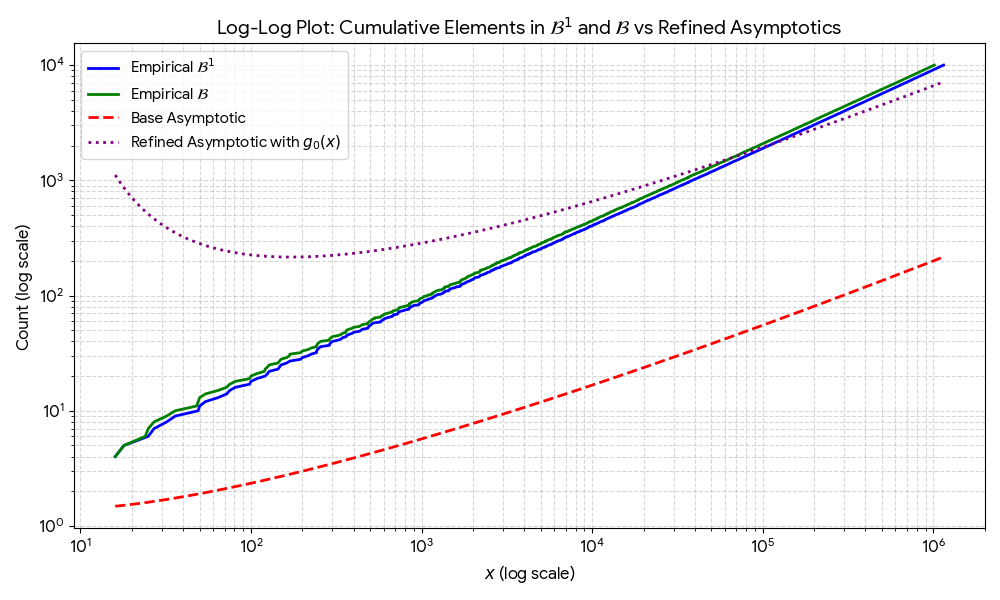}
    \caption{A log-log plot of $\#({\mathcal B}^1 \cap [1,x])$ and $\#({\mathcal B} \cap [1,x])$, against the base prediction in \Cref{b-asym}(i), as well as the more refined estimate in \eqref{g0}.  The discrepancy between the two approximations is quite large in this range due to the slow decay of $g_0(x) = O(\log_3 x / \log_2 x)$.  (Image generated by Gemini using data from the OEIS.)}
    \label{fig-bad}
\end{figure}

The first few elements of ${\mathcal B}$ are
\begin{equation}
4, 8, 9, 16, 18, 24, 25, 27, 32, 36, 48, 49, 50, 54, 64, 72, \dots \tag{${\mathcal B}$}
\end{equation}
(\href{https://oeis.org/A388654}{OEIS A388654}).  As one can see from this (or \Cref{fig-cumulative}), this set is very similar to ${\mathcal B}^1$, but contains a few more elements; for instance, if $p$ is an odd prime, then $\{p^2-1,p^2\}$ can be checked to be bad, and hence $p^2-1$ will lie in ${\mathcal B}$, though it is unlikely to lie in ${\mathcal B}^1$.  The first few elements of ${\mathcal B}$ that do not lie in ${\mathcal B}^1$ are
\begin{equation}
 24, 48, 120, 168, 360, 528, 840, 960, 1155, 1368, 1680, 1683, \dots \tag{${\mathcal B} \backslash {\mathcal B}^1$}
\end{equation}
(\href{https://oeis.org/A387054}{OEIS A387054}).  For instance $1683 = 41^2+2 = 3^2 \times 11 \times 17$ lies in this sequence because it is not divisible by the square of its largest prime factor, but
$$1681 \times 1682 \times 1683 = 2 \times 3^2 \times 11 \times 17 \times 29^2 \times 41^2$$
is.

In \cite[Fact 4]{erdos-graham}\footnote{In \cite[p. 73]{eg} this was reported as $\#({\mathcal B} \cap [1,x]) > x^{1-\eps}$, but this appears to be a typographical error.} the asymptotic
\begin{equation}\label{fact4}
\#({\mathcal B} \cap [1,x]) = o(x)
\end{equation}
as $x \to \infty$ was established.  This was later improved in \cite[Theorem 2]{luca} to
\begin{equation}\label{b-bound}
\#({\mathcal B} \cap [1,x]) \ll \frac{x}{\exp((c+o(1)) \log^{1/4} x \log_2^{3/4} x)}
\end{equation}
for some constant $c>0$.

Our main result for this set is the following further improvement.

\begin{theorem}[Asymptotic for bad sets]\label{bad-thm}  One has
$$   \#(({\mathcal B} \backslash {\mathcal B}^1) \cap [1,x]) \ll \frac{\#({\mathcal B}^1 \cap [1,x])}{\log^{1-o(1)} x} $$
as $x \to \infty$.  In particular, \eqref{b-conj} holds, and
$$ \#({\mathcal B} \cap [1,x]) = \frac{x}{z^{2+o(1)}}.$$
\end{theorem}

This affirms a conjecture in \cite[p. 73]{eg} (see also \cite[Problem 380]{bloom}). This is the most difficult result of this paper; we establish it in \Cref{main-sec}.

\subsection{Very bad intervals}

The first few elements of the set $\mathcal{VB}^1$ of powerful numbers are
\begin{equation}\label{vbad}
1, 4, 8, 9, 16, 25, 27, 32, 36, 49, 64, 72, 81, 100, \dots \tag{$\mathcal{VB}^1$}
\end{equation}
(\href{https://oeis.org/A001694}{OEIS A001694}).  The set $\mathcal{VB}^1$ is contained in ${\mathcal B}^1$ (and hence ${\mathcal B}$); however it is significantly smaller.  One can equivalently view a powerful number as the product $a^2 b^3$ of a perfect square $a^2$ and a perfect cube $b^3$ with $b$ squarefree, with every powerful number uniquely representable in this form; see, e.g., \cite[Lemma 2.3]{aktas}.  Using this representation, it was shown in \cite{erdos-szekeres} (see also \cite{golomb}) that
\begin{equation}\label{count-power}
 \#(\mathcal{VB}^1 \cap [1,x]) \sim \frac{\zeta(3/2)}{\zeta(3)} \sqrt{x}
\end{equation}
where the \emph{Erd\H{o}s--Szekeres constant} $\frac{\zeta(3/2)}{\zeta(3)}$ can be computed explicitly as
$$\frac{\zeta(3/2)}{\zeta(3)}  = 2.1732543\dots.$$
(\href{https://oeis.org/A090699}{OEIS A090699}). In \cite[Theorem 3]{bateman-grosswald} the refined asymptotic
\begin{equation}\label{count-power-2}
 \#(\mathcal{VB}^1 \cap [1,x]) = \frac{\zeta(3/2)}{\zeta(3)} \sqrt{x} + \frac{\zeta(2/3)}{\zeta(2)} x^{1/3} + O(x^{1/6})
\end{equation}
was established, which is an exceptionally good fit in practice; see \Cref{fig-vb}.

\begin{figure}
    \centering
\includegraphics[width=0.75\textwidth]{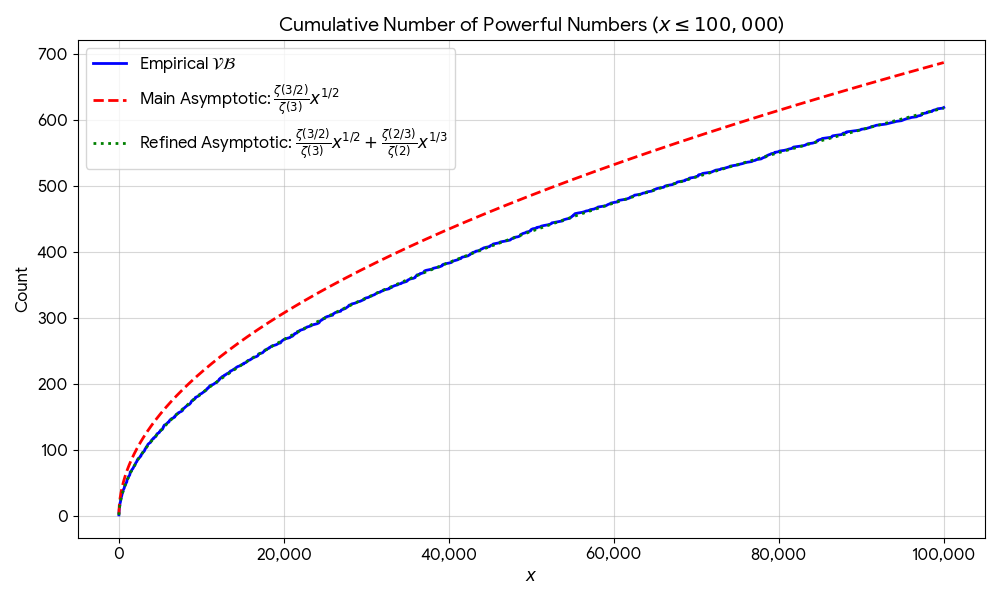}
    \caption{A plot of $\#(\mathcal{VB}^1 \cap [1,x])$ (which is numerically identical to $\#(\mathcal{VB} \cap [1,x])$) against the predictions in \eqref{count-power}, \eqref{count-power-2}, with the latter being an exceptionally good fit.  (Image generated by Gemini using data from the OEIS.)}
    \label{fig-vb}
\end{figure}

The set $\mathcal{VB}$ contains $\mathcal{VB}^1$, but numerically we know of no elements of $\mathcal{VB}$ that do not already lie in $\mathcal{VB}^1$; thus the first few elements of $\mathcal{VB}$ are also given by \eqref{vbad}.

There exist very bad intervals $\{N-1,N\}$ of length two, that is to say consecutive powerful numbers.  As observed by Mahler, any solution to the Pell equation $x^2-8y^2=1$ will give such a very bad pair $\{x^2-1,x^2\}$, although there are other examples, such as the example $\{12167, 12168\} = \{23^3, 2^3 \times 3^2 \times 13^2\}$ of Golomb \cite{golomb} and the examples $\{ 7^3 x^2-1, 7^3 x^2\}$ of Walker \cite{walker} arising from solutions to the Pell-type equation $7^3 x^2 - 3^3 y^2 = 1$.  See \cite[p. 31]{erdos76},\cite[p. 80]{eg}, \cite[Problem \#365]{bloom} for further discussion.  The first few $N$ for which $N-1$, $N$ are both powerful are
\begin{equation}
9, 289, 676, 9801, 12168, 235225, 332929, 465125, 1825201, 11309769, \dots \tag{$\mathcal{VB}^1 \cap (\mathcal{VB}^1+1)$}
\end{equation}
 (\href{https://oeis.org/A078326}{OEIS A078326}). In \cite[p. 31]{erdos-consecutive}, \cite[p. 234]{ribenboim}, \cite[p. 68]{eg}, \cite[Problem \#365]{bloom} it is conjectured that
\begin{equation}\label{erd-consec}
 \#(\mathcal{VB}^1 \cap (\mathcal{VB}^1+1) \cap [1,x]) \ll \log^{O(1)} x
\end{equation}
but at present the best known upper bound is
\begin{equation}\label{vbvb}
 \#(\mathcal{VB}^1 \cap (\mathcal{VB}^1+1) \cap [1,x]) \ll x^{2/5},
\end{equation}
with the bound improving to $x^{o(1)}$ assuming the $abc$ conjecture; see \cite{aktas}, \cite{chan}, as well as Remark \ref{rem-divisor} below. We also mention the work of Blomer \cite{blomer} lower bounding the size of the set $\mathcal{VB}^1 + \mathcal{VB}^1$ of numbers that are the sum of two powerful numbers.

 It was conjectured in Erd\H{o}s and Selfridge \cite[p. 73]{eg} (see also \cite[Problem \#137]{bloom}) that very bad intervals have length at most two, which (by the coprimality of consecutive integers) would confirm that $\mathcal{VB} = \mathcal{VB}^1$.  This conjecture remains open; it implies the weaker, but still open, conjecture of Erd\H{o}s \cite{erdos76} (see also \cite[Problem \#364]{bloom}) that there are no three consecutive powerful numbers.  We remark that the $abc$ conjecture implies that the number of triples of consecutive powerful numbers is finite: see, e.g., \cite[Exercise 1.3.6]{murty}.

Our main result for the set $\mathcal{VB}$ is the following variant of \eqref{vbvb}:

\begin{theorem}[Asymptotic for very bad sets]\label{verybad-thm}  One has
$$   \#((\mathcal{VB} \backslash \mathcal{VB}^1) \cap [1,x]) \ll x^{\frac{2}{5}+o(1)}$$
as $x \to \infty$.  In particular, since $\frac{2}{5} < \frac{1}{2}$, \eqref{vb-conj} holds, and
$$ \#(\mathcal{VB} \cap [1,x]) \sim \frac{\zeta(3/2)}{\zeta(3)} \sqrt{x}.$$
\end{theorem}

This answers a conjecture from \cite[p. 345]{erdos-graham}, \cite[p. 73]{eg} (also mentioned in the commentary to \cite[Problem 380]{bloom}).  We prove this result in \Cref{verybad-sec}.  The arguments are similar to those used to prove \eqref{vbvb}, and any argument that improves that latter bound would likely be applicable to improve \Cref{verybad-thm} as well.

\subsection{Type \texorpdfstring{$F_3$}{F3} intervals}

The first few elements of the set ${\mathcal F}_3^1$ are
\begin{equation}
 4, 6, 8, 9, 16, 18, 20, 24, 25, 28, 30, 32, 35, 36, 45, 49, 50, \dots \tag{${\mathcal F}_3^1$}
\end{equation}
(\href{https://oeis.org/A387186}{OEIS A387186}).
It contains the square numbers $n^2$ larger than $1$, but for each natural number $a$, also contains all sufficiently large square multiples $s(a!)n^2$ of the squarefree part $s(a!)$ of $a!$.  The squarefree parts $s(a!)$ are
\begin{equation}
    1, 2, 6, 6, 30, 5, 35, 70, 70, 7, 77, 231, 3003, \dots \tag{$s(a!)$}
\end{equation}
(\href{https://oeis.org/A055204}{OEIS A055204}); from the prime number theorem one can verify that $s(a!)$ grows at least exponentially fast in $a$. A simple counting argument then establishes the asymptotic
\begin{equation}\label{f31}
   \#({\mathcal F}_3^1 \cap [1,x]) \sim c_3^1 \sqrt{x}
\end{equation}
where the constant $c_3^1$ is defined as
$$ c_3^1 \coloneqq \sum_{s \in S} \frac{1}{s^{1/2}}$$
where $S \coloneqq \{ s(a!): a \geq 1\} = \{1, 2, 5, 6, 7, 30, \dots\}$ is the set of values of $s(a!)$ (not counting multiplicity), and can be numerically computed as
$$ c_3^1 = 3.709751\dots$$
(\href{https://oeis.org/A389117}{OEIS A389117}).  Note from \Cref{classical}(ii) that the sequence $s(a!)$ only repeats at the square numbers, so we can also write
$$ c_3^1 = 1 + \sum_{a \geq 2: a \neq n^2 \forall n} \frac{1}{s(a!)^{1/2}}.$$

\begin{figure}
    \centering
\includegraphics[width=0.75\textwidth]{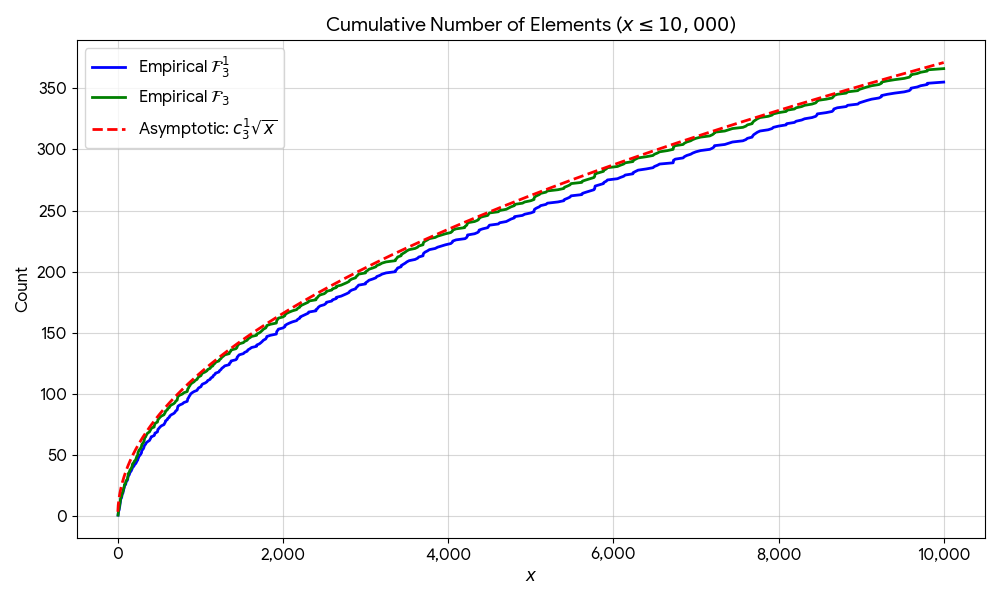}
    \caption{A plot of $\#({\mathcal F}_3^1 \cap [1,x])$ and $\#({\mathcal F}_3 \cap [1,x])$ against the prediction in \eqref{f31}.  (Image generated by Gemini using data from the OEIS.)}
    \label{fig-f3}
\end{figure}

The set ${\mathcal F}_3$ contains ${\mathcal F}_3^1$, but is slightly larger empirically, as one can see from \Cref{fig-cumulative}. The first few elements are
\begin{equation}
4, 6, 8, 9, 10, 16, 18, 20, 24, 25, 28, 30, 32, 35, 36, 45, 49, 50, \dots \tag{${\mathcal F}_3$}
\end{equation}
(\href{https://oeis.org/A388851}{OEIS A388851}); see also \cite{dujella}.  As observed in \cite[(11)]{erdos-graham}, if $a! = uv$ is any factorization of a factorial and $x,y$ solve the Pell-type equation $ux^2-vy^2=1$ then
$$ a! (vy^2) (vy^2+1) = (a! xy)^2$$
is a perfect square, and so for $y$ sufficiently large $\{vy^2,vy^2+1\}$ is of type $F_3$, so that $vy^2+1$ will lie in ${\mathcal F}_3$ if $y$ is large enough.  For instance, using the factorization $3! = 3 \times 2$ and the Pell equation solution $3 \times 9^2 - 2 \times 11^2 = 1$ we see that
$$ 3! \times 242 \times 243 = 352836 = 594^2 = (3! \times 9 \times 11)^2,$$
so that $\{242,243\}$ is of type $F_3$ and $243$ lies in ${\mathcal F}_3$, even though it does not lie in ${\mathcal F}_3^1$.  To see the latter claim, observe that for every $a \geq 2$, \Cref{prime-interval}(i) below supplies a prime $p \in (a/2,a]$ that divides $a!$ exactly once and hence divides $s(a!)$; thus $s(a!) = s(243) = 3$ would force $a \leq 5$, but $s(a!) \in \{1,2,6,30\}$ for $1 \leq a \leq 5$.  (The same criterion shows that $10 \not \in {\mathcal F}_3^1$, since $s(10)=10$.)  Given the sparsity of solutions to Pell's equation, the total number of contributions to ${\mathcal F}_3$ arising from this construction would be expected to be small compared with ${\mathcal F}_3^1$, though still infinite.

The above constructions only give type $F_3$ intervals of length one or two.  In \cite{erdos-graham}, two sporadic type $F_3$ intervals of length three, namely $\{8,9,10\}$ and $\{48,49,50\}$, were noted,
\begin{align*}
s(8 \times 9 \times 10) &= 5 = s(6!) \\
s(48 \times 49 \times 50) &= 6 = s(3!) = s(4!).
\end{align*}
and they asked whether any further examples exist. In \cite{dujella} the additional examples
\begin{align*}
s(322 \times 323 \times 324) &= 104006 = s(26!)\\
s(350 \times 351 \times 352) &= 3003 = s(13!)\\
s(440 \times 441 \times 442) &= 12155 = s(18!)\\
s(2736 \times 2737 \times 2738) &= 104006 = s(26!)
\end{align*}
were located, and they also showed that this is the complete list of such intervals of length three with $a \leq 100$.  It was asked in \cite[p. 345]{erdos-graham} whether there were any type $F_3$ intervals of length greater than $3$; this remains open.

As a consequence of \cite[Theorem 2]{erdos-graham} it was shown that
$$ \# ({\mathcal F}_3 \cap [1,x]) = o(x)$$
as $x \to\infty$, and in \cite[Theorem 1]{luca} this was improved to
$$
\# ({\mathcal F}_3 \cap [1,x]) \ll \frac{x}{\exp((c+o(1)) \log^{1/4} x \log_2^{3/4} x)}
$$
for some fixed $c>0$.  In \Cref{f3-sec} we establish the following improvement:

\begin{theorem}[Near-asymptotic for ${\mathcal F}_3$]\label{f3-thm}  One has
$$   \#(({\mathcal F}_3 \backslash {\mathcal F}_3^1) \cap [1,x]) \ll x^{\frac{1}{2}+o(1)}$$
as $x \to \infty$.
\end{theorem}

This comes close to establishing the claim \eqref{f3-conj}, which was conjectured in \cite[p. 346]{erdos-graham}.  To complete the proof of this conjecture one needs to improve the $x^{1/2+o(1)}$ bound to $o(x^{1/2})$, but to do this via our methods requires breaking the infamous ``square root barrier'' in sieve theory, which we do not know how to do.  On the other hand, from \Cref{f3-thm} and \eqref{f31} we have
$$\#({\mathcal F}_3 \cap [1,x]) = x^{\frac{1}{2}+o(1)}$$
giving a weak answer to the $k=3$ case of \cite[Problem \#374]{bloom}.

As a corollary of our proof methods, we also obtain a bound for the number of solutions to \eqref{f3-eq}, which we also prove in \Cref{f3-sec}:

\begin{theorem}[Counting solutions to a factorial equation]\label{thm-f3-eq}  The number of solutions to \eqref{f3-eq} with $1 \leq a_1 < a_2 < a_3 \leq x$ is $x^{\frac{1}{2}+o(1)}$ as $x \to \infty$.
\end{theorem}

In view of \eqref{f3-conj} and \eqref{f31} it is natural to conjecture that the number of such solutions is in fact $\sim c_3^1 \sqrt{x}$; certainly the lower bound is true thanks to \eqref{f31}.

\begin{remark} In view of \cite{berczes}, one could try to adapt these results to study solutions to the equation $a_1! a_2! a_3! = n^k$ for fixed $k > 2$ and $a_1 < a_2 < a_3$, but our arguments use the specific choice $k=2$ at various junctures (for instance, in order to ignore the square denominator in \eqref{nhint}), so it is not immediately obvious to the author whether the methods extend to this setting.
\end{remark}

\subsection{Methods of proof: the very bad and \texorpdfstring{$F_3$}{F3} intervals}

We now discuss the methods of proof of \Cref{verybad-thm} and \Cref{f3-thm}.  We begin with \Cref{verybad-thm}.  If $\{N+1,\dots,N+H\}$ is very bad, and $p$ is a prime in the range $H < p \leq 2H$, then $p$ can only divide one of the terms in \eqref{nhint}, in which case it must also divide it twice by \Cref{def-bad}(ii).  This causes the fractional parts of $\frac{N}{p}$ and $\frac{N}{p^2}$ to be highly non-equidistributed for $p$ in this range, which when combined with some equidistribution results of Vinogradov (see \Cref{vinogradov}) can be used to show that $H$ is quite small, in particular $H = N^{o(1)}$.  The powerful nature of the product \eqref{nhint} can then be used to show that two elements in the interval $\{N+1,\dots,N+H\}$ are powerful up to small prime factors, thus creating a linear relation of the form $an+h = bm$ where $a,b,h$ are small (of polynomial size in $H$) and $n,m$ are powerful.  The main remaining task is then to upper bound the number of solutions to such equations.  Expressing powerful numbers as the product of a square and a cube, matters reduce to obtaining good upper bounds for the number of solutions to a hyperbola $an^2 + h = bm^2$ where $a,b,h$ are not too large.  This can be achieved through a version of the divisor bound for quadratic fields. The arguments here are similar to those in \cite{aktas}.

The proof of \Cref{f3-thm} is analogous -- in particular one can also ensure that the length $H$ of the interval is small -- but an additional parameter $a$ now appears, namely the quantity for which \eqref{nhint} shares its squarefree component with $a!$.  This forces all the primes between $a/2$ and $a$ to divide \eqref{nhint}, which leads to the bound\footnote{When dealing with $k>3$ equations, such as $a_1! a_2! (N+1) \dots (N+H) = m^2$, the author was unable to achieve a similarly useful upper bound on $a_1$ or $a_2$, and as a consequence was unable to obtain any new results for such equations.} $a \ll H \log N$.  If $a = o(H \log N)$ then it turns out that the previous arguments can be adapted to handle this case also, as the coefficients of the hyperbolae that appear are still quite small; and a similar method also treats the case when $H$ is bounded (the point now being that the coefficients of the hyperbola are now quite smooth).  But the previous method breaks down in the regime where $H$ is large and $a \asymp H \log N$.  Here we resort to a less efficient method, using the aforementioned fact that the primes $p$ between $a/2$ and $a$ divide \eqref{nhint} to impose some restrictive congruence restrictions on $N$ modulo each such prime $p$.  Another application of the large sieve then gives the $x^{1/2+o(1)}$ bound; unfortunately this is close to the limit of current sieve-theoretic methods, and some other technique would be needed to handle this case if one wished to obtain the full conjecture \eqref{f3-conj}.

\subsection{Methods of proof: the bad intervals}

We now discuss the methods of proof of \Cref{bad-thm}, which is the most difficult result to establish in this paper, and requires different techniques than those used to handle the very bad or $F_3$ intervals.  It turns out that bad intervals $\{N+1,\dots,N+H\}$ of a large length $H$ (e.g., $H \geq \log^{20} x$) can be shown to have a relatively small contribution by sieve methods, together with known facts about the scarcity of large prime gaps.  The main difficulty is with the short bad intervals.  A typical case arises when $H=2$ and $N+1 = p_0^2 m$ for some prime $p_0 = z^{1+o(1)}$, and some $p_0$-smooth number $m$ less than $x/p_0^2$ such that $p_0^2 m + 1$ is also $p_0$-smooth; then $p_0^2 m+1$ will be contained in the bad interval $\{p_0^2 m, p_0^2 m+1\}$, but will most likely not be divisible by the square of its largest prime factor.  The problem of controlling the number of examples $p_0^2 m + 1$ of this form is similar to the familiar problem of counting consecutive smooth numbers; however, the level $p_0 = z^{1+o(1)}$ of smoothness here is too large for tools such as St{\o}rmer's theorem \cite{stormer} to be applicable, while simultaneously being too small for results on binary correlations of multiplicative functions \cite{teravainen} to be usable.  Instead, we use an ``anti-sieve'', taking advantage of the fact that smooth numbers have an \emph{elevated} chance of being divisible by small primes (as opposed to the classical sieve-theoretic situation where one often wishes to count rough numbers not divisible by such a prime).

For sake of this informal discussion, let us assume that $p_0^2 m+1$ is squarefree and comparable to $x$.  Then from the fundamental theorem of arithmetic, if $p_0^2 m + 1$ is to be $p_0$-smooth, we must have
\begin{equation}\label{lpx}
 \sum_{p < p_0} 1_{p|p_0^2 m + 1} \log p = \log(p_0^2 m + 1) \asymp \log x.
\end{equation}
This can be compared with Mertens' theorem, which gives
$$ \sum_{p < p_0} \frac{1}{p} \log p \asymp \log z = \log^{1/2+o(1)} x.$$
If we can show that the events $1_{p|p_0^2 m + 1}$ are ``approximately independent'' as $m$ varies, one can then hope to show that the event
\eqref{lpx} is rare, by using moment methods.
To do this, we utilize the anatomy of $p_0$-smooth numbers to factor (most) $p_0$-smooth numbers $m$ as\footnote{The parameter $1000$ here needs to be sufficiently large in order to obtain enough ``mixing'' from a certain application of 

the Burgess character sum bound later in the argument.  We chose to use explicit numbers here to avoid introducing a certain number of ``sufficiently large'' or ``sufficiently small'' parameters and exponents in the arguments.}
$$ m = p_1 \dots p_{1000} m'$$
where the $p_1,\dots,p_{1000}$ are primes with
$$ z^{1-o(1)} \leq p_{1000} \leq \dots \leq p_1 \leq p_0 \leq z^{1+o(1)}$$
and $m'$ is $p_{1000}$-smooth.  It is convenient to view $m'$ as being fixed, and the primes $p_i$ for $i=0,\dots,1000$ as free parameters ranging over various dyadic intervals $[P_i,2P_i)$.  Applying the second\footnote{Actually, we only use the second moment method to deal with relatively large primes $p$, e.g., $p > z^{1/100}$.  For smaller $p$, we need to use higher moments to get satisfactory bounds.} moment method, the question then reduces to that of understanding the distribution of products $p_0^2 p_1 \dots p_{1000}$ with respect to moduli such as $p$ or $pp'$, which by standard Dirichlet character expansions then comes down to the task of upper bounding prime character sums
$$ \frac{|\sum_{P_i \leq p_i < 2P_i} \chi(p_i)|}{\sum_{P_i \leq p_i < 2P_i} 1},$$
where $\chi$ is a primitive Dirichlet character with a conductor such as $p$ or $pp'$, and the prime $\mathbf{p}_i$ is drawn uniformly at random amongst the primes in a dyadic interval $[P_i,2P_i)$ for some $P_i = z^{1+o(1)}$.

Here we run into the familiar issue that $\chi$ may be related to a Siegel zero that causes such prime character sums to be extremely large.  However, due to the repulsion phenomenon, we expect such exceptional characters to be rare.  We will therefore close the argument by quantifying the rarity of such exceptional characters.  As it turns out, it will be inconvenient to work directly with Siegel zeroes and Dirichlet $L$-functions at our choice of scales; instead, we shall use almost orthogonality methods (the Bombieri--Hal\'asz--Montgomery inequality), where the required almost orthogonality will be provided by the well-known Burgess bound on character sums, combined with a standard sieve to restrict to almost primes to avoid unwanted logarithmic factors in the estimates.

\subsection{Acknowledgments and AI tool disclosure}

The author was supported by the James and Carol Collins Chair, the Mathematical Analysis \& Application Research Fund, and by NSF grant DMS-2347850, and is particularly grateful to recent donors to the Research Fund.  We thank Tim Browning for some references, Dan Asimov and Antoine Deleforge for a correction, Nat Sothanaphan for proofreading an initial version of this text (using GPT-5.4 Thinking), and the anonymous referee for detailed comments.  An anonymous commenter on the author's blog provided additional references for \Cref{squarecount}. 

Gemini was used to generate the figures in the text, Github Copilot was used for text autocomplete, Claude Code was used for typesetting, and ChatGPT Pro was used for proofreading and for supplying the proof of \Cref{squarecount}.  Outside of these AI tool uses, the text of this paper was human generated.

\section{Notation and basic estimates}\label{notation-sec}

Throughout this paper we take $x$ to be a sufficiently large real number.  We use the following asymptotic notation:
\begin{itemize}
\item[(i)] $X \ll Y$, $Y \gg X$, or $X = O(Y)$ means there exists a constant $C>0$ such that $|X| \leq CY$.  If we need the implied constant $C$ to depend on a parameter, we will subscript the asymptotic notation accordingly; for instance, $X \ll_k Y$ denotes an estimate of the form $|X| \leq C_k Y$.
\item[(ii)] $X \asymp Y$ is shorthand for $X \ll Y \ll X$.
\item[(iii)] $o(1)$ denotes a quantity that goes to zero in some specified asymptotic regime, such as $x \to \infty$.
\item[(iv)] $X \sim Y$ is shorthand for $X = (1+o(1))Y$.
\end{itemize}

For the purposes of sums, products, and set builder notation, the symbol $p$ is always understood to range over primes.
If $E$ is a finite set, we use $\# E$ to denote its cardinality; if instead $E$ is a measurable subset of $\R$, we use $|E|$ to denote its Lebesgue measure.  In either case we use $1_E$ to denote the indicator function of $E$, thus $1_E(x)$ equals $1$ when $x \in E$ and zero otherwise.  If $x$ is a real number, we use $\lfloor x \rfloor$ and $\{x\}$ to denote the integer and fractional parts of $x$ respectively, and write $e(x) \coloneqq e^{2\pi i x}$.  If $a_1,\dots,a_n$ are natural numbers, we use $[a_1,\dots,a_n]$ to denote their least common multiple; and if $q$ is a natural number, we use $\phi(q)$ to denote the Euler totient function of $q$.

Random variables and events will be denoted using boldface symbols.
We use the usual probabilistic notation $\P(\mathbf{E})$, $\E(\mathbf{X})$, $\Var(\mathbf{X}) = \E |\mathbf{X}|^2 - |\E \mathbf{X}|^2$, $\Cov(\mathbf{X},\mathbf{Y}) = \E \mathbf{X} \mathbf{Y} - (\E \mathbf{X}) (\E \mathbf{Y})$ to denote the probability of an event $\mathbf{E}$, the mean and variance of a real random variable $\mathbf{X}$, and the covariance of two real random variables $\mathbf{X}$, $\mathbf{Y}$ respectively.

\subsection{Smooth numbers}

Let $y \geq 1$.  We say that a natural number $n$ is \emph{$y$-smooth} if all its prime factors are at most $y$, and write $\Psi(x,y)$ to denote the number of smooth numbers up to $x$.  This quantity is well studied; see for instance \cite{granville}, \cite{hildebrand}, or \cite{tenenbaum}.  For our purposes, we will just need the following relatively crude estimates:

\begin{proposition}[Smooth number estimates]\label{smoot}\
    \begin{itemize}
        \item[(i)]  If $y = z^{\alpha+o(1)}$ for some $\alpha \asymp 1$, then $\Psi(x,y) = \frac{x}{z^{1/\alpha+o(1)}}$.  Furthermore, $\Psi(cx,y) \asymp \Psi(x,y)$ for any $c \asymp 1$ (where the implied constants here can depend on $c,\alpha$).
        \item[(ii)]  If $y = \log^{A+o(1)} x$ for some $1 < A \ll 1$, then $\Psi(x,y) = x^{1-1/A+o(1)}$.
    \end{itemize}
    The same estimtaes hold if $x$ is replaced by any $x'$ of the form $x' = x^{1-o(1)}$ (leaving the definition of $z$ in terms of $x$ unchanged).  
\end{proposition}

\begin{proof}  For part (i), see \cite[(1.15), (3.24)]{granville}; for part (ii), see \cite[(1.14)]{granville}. We refer the reader to \cite{tenenbaum} for more sophisticated estimates of this type.
\end{proof}

Observe from \Cref{def-sets} that a number $n$ lies in ${\mathcal B}^1$ if and only if it is of the form $n = p^2 m$ for some prime $p$ and $p$-smooth $m$, with each such $n$ having exactly one representation of this form.  This gives the exact formula
\begin{equation}\label{box}
\#({\mathcal B}^1 \cap [1,x])  = \sum_{p \leq \sqrt{x}} \Psi\left(\frac{x}{p^2}, p\right).
\end{equation}

As a first application of these estimates, we can now quickly prove \Cref{b-asym}.

\begin{proof}[Proof of \Cref{b-asym}] By Mertens' theorem we have
$$\sum_{p \leq \sqrt{x}} \frac{1}{p} \asymp \log_2 x = z^{o(1)}$$
and also
$$\sum_{p = z^{1+o(1)}} \frac{1}{p} = z^{o(1)}.$$
To prove (i), it will suffice by \eqref{box} to show that
\begin{equation}\label{e1}
 \Psi\left(\frac{x}{p^2}, p\right) \ll \frac{1}{z^{2+o(1)}} \frac{x}{p}
\end{equation}
for all $p \leq \sqrt{x}$, with the matching lower bound
\begin{equation}\label{e2}
 \Psi\left(\frac{x}{p^2}, p\right) \gg \frac{1}{z^{2+o(1)}} \frac{x}{p}
\end{equation}
for all $p$ in the range
\begin{equation}\label{p-range}
p = z^{1+o(1)}.
\end{equation}
In the range $p \geq z^2$ the claim follows from the trivial bound
$$ \Psi\left(\frac{x}{p^2}, p\right) \leq \frac{x}{p^2} \leq \frac{1}{z^2} \frac{x}{p}$$
while for $p \leq z^{1/2}$ the claim follows from \Cref{smoot}(i) since
$$ \Psi\left(\frac{x}{p^2}, p\right) \leq \Psi\left(\frac{x}{p^2}, z^{1/2}\right) \leq \frac{x/p^2}{z^{2+o(1)}}.$$
In the remaining cases, we have $p = z^{\alpha+o(1)}$ for some $1/2 \leq \alpha \leq 2$, and then by \Cref{smoot}(i)
$$ \Psi\left(\frac{x}{p^2}, p\right) \leq \frac{x/p^2}{z^{1/\alpha+o(1)}} = \frac{1}{z^{\alpha+\frac{1}{\alpha}+o(1)}} \frac{x}{p}$$
and \eqref{e1} then follows in this case from the arithmetic mean-geometric mean inequality; the claim \eqref{e2} meanwhile follows from the $\alpha=1$ case of \Cref{smoot}(i).  This proves (i). The above analysis, when combined with the second part of \Cref{smoot}(i), also gives the claim (ii).
\end{proof}

\begin{remark}
Note that this analysis also suggests that the primes $p$ in the range \eqref{p-range} should play a dominant role in the analysis of bad intervals.  Sharper estimates on the quantity $\# ({\mathcal B}^1 \cap [1,x])$ may be found in \cite{ivic} (which uses $T_0(x)$ to denote $\# ({\mathcal B}^1 \cap [1,x])$).
\end{remark}

\subsection{Primes in short intervals}

Now we collect some standard results about primes in short intervals.

\begin{proposition}[Primes in short intervals]\label{prime-interval}\
    \begin{itemize}
        \item[(i)]  (Bertrand's postulate) If $N \geq 2$, then there exists a prime $p$ in the range $N/2 < p \leq N$.
        \item[(ii)]  (Baker--Harman--Pintz) If $N \geq 0$, $H \geq 1$ and $\{N+1,\dots,N+H\}$ does not contain any prime, then $H \ll N^{0.525}$.
        \item[(iii)] (Application of Guth--Maynard) If $\theta > \frac{2}{15}$, then there exists $c>0$ such that as $x \to \infty$, the set of real numbers $0 \leq x' \leq x$ for which $[x', x'+x^\theta]$ does not contain a prime has measure $O(x^{1-c+o(1)})$.
    \end{itemize}
\end{proposition}

\begin{proof} Part (i) is classical.  For (ii), see \cite{bhp}.  For part (iii), we apply the recent zero-density estimate of \cite{guth-maynard} (when combined with the arguments in\footnote{Much stronger results are known assuming the Riemann hypothesis: see \cite{hb}.} \cite{gafni-tao} or \cite[Theorem 2(i)]{bazzanella}).
\end{proof}

\begin{remark}\label{rem-bounds}  The precise upper bound of $N^{0.525}$ in \Cref{prime-interval}(ii) is not critical for our purposes; indeed even the far weaker bound of $o(N / \log N)$, obtainable from the prime number theorem with classical error term, would have already sufficed.  However, it will be important in the proof of \Cref{bad-thm} that \Cref{prime-interval}(iii) is available for some $\theta < 1/6$, which only became possible with the recent work of Guth and Maynard.  There are results similar to \Cref{prime-interval}(iii) that are available for some $\theta < 1/6$ if one is willing to weaken the bound on the exceptional set; for instance, the results\footnote{We thank an anonymous commenter on the author's blog for this remark.} in \cite{harman} work for any $\theta > 1/10$ but with only a bound of $O(x/\log x)$ on the exceptional set, which unfortunately is not sufficient for our purposes.
\end{remark}

\subsection{Equidistribution estimates for primes}

Vinogradov famously gave non-trivial bounds for exponential sums such as $\sum_{P < p \leq 2P} e(T f(p/P))$ for various smooth phase functions $f$ and relatively large values of frequency $T$ (e.g., $T$ as large as $\exp(\log^{3/2-\eps} P)$).  This gives rather strong control on the equidistribution of the fractional parts of $T f(p/P)$.  We require the following version of this equidistribution estimate.

\begin{theorem}[Vinogradov-type equidistribution estimate]\label{vinogradov}  Let $\eps>0$, $P \geq 2$, $A>0$, and $M_1, M_2 = O(\exp(\log^{3/2-\eps} P))$.  Let $W \colon \R^2 \to \R$ be a smooth, $\Z^2$-periodic function.  Then for any interval $I$ in $[P,2P]$ and $j \geq 1$, one has
$$ \sum_{p \in I} W\left(\frac{M_1}{p}, \frac{M_2}{p^j}\right) = \int_I W\left(\frac{M_1}{t}, \frac{M_2}{t^j}\right) \frac{dt}{\log t} + O_{\eps,A}\left(\frac{\|W\|_{C^3} P}{\log^A P}\right)$$
where $\|W\|_{C^3} \coloneqq \max_{0 \leq i \leq 3} \sup_{y \in \R^2} |\nabla^i W(y)|$.
\end{theorem}

\begin{proof} See \cite[Proposition 1.12]{singmaster}.  Related estimates can also be found in \cite{granville-ramare}.
\end{proof}

\begin{remark}
    In our applications we will only need the case when $M_1=M_2$ and $j=2$.
\end{remark}

\subsection{The large sieve}

A standard form of a large sieve estimate involves bounding the size of an interval $I$ after one removes a certain number $\omega(p)$ of residue classes modulo $p$ for various primes $p$; see for instance \cite[Corollary 1]{montgomery-vaughan}. We begin by recalling an uncertainty principle of Montgomery \cite{montgomery-vaughan}, \cite{selberg}:

\begin{lemma}[Montgomery uncertainty principle]\label{montgomery} Let ${\mathcal Q}$ be a finite collection of primes $q$.  Let $f \colon \N \to \C$ be a finitely supported function, such that for each $q \in {\mathcal Q}$, there are $\omega(q)$ residue classes modulo $q$ on which $f$ vanishes.  Then for any distinct $q_1,\dots,q_k \in {\mathcal Q}$, one has
$$ \sum_{a \in \Z/q_1 \dots q_k\Z: q_1,\dots,q_k \nmid a} \left|\sum_n f(n) e\left(-\frac{an}{q_1 \dots q_k}\right)\right|^2 \geq \left(\prod_{j=1}^k \frac{\omega(q_j)}{q_j-\omega(q_j)}\right) |\sum_n f(n)|^2,$$
where $e(\theta) \coloneqq e^{2\pi i \theta}$.
\end{lemma}

\begin{proof}  For each $q \in {\mathcal Q}$, let $S_q$ denote the union of the $\omega(q)$ residue classes modulo $q$ on which $f$ vanishes, and set
$$ \nu(n) \coloneqq \prod_{j=1}^k (\omega(q_j) - q_j 1_{S_{q_j}}(n)).$$
Then $\nu$ is $q_1 \dots q_k$ periodic and
$$ \frac{1}{q_1 \dots q_k} \sum_{n \in \Z/q_1 \dots q_k\Z} \nu(n)^2 = \prod_{j=1}^k \omega(q_j) (q_j - \omega(q_j)).$$
Thus by Fourier expansion and Plancherel's theorem one can write
$$ \nu(n) = \sum_{a \in \Z/q_1 \dots q_k\Z} c_a e\left(\frac{an}{q_1 \dots q_k}\right)$$
for some coefficients $c_a$ with
\begin{equation}\label{ca2}
 \sum_{a \in \Z/q_1 \dots q_k\Z} |c_a|^2 =  \prod_{j=1}^k \omega(q_j) (q_j - \omega(q_j)).
\end{equation}
Since each of the factors $\omega(q_j) - q_j 1_{S_{q_j}}$ of $\nu$ is $q_j$-periodic with mean zero, we see that $c_a$ vanishes if $a$ is divisible by any of the $q_1,\dots,q_k$.

On the support of $f$, we have $\nu(n) = \prod_{j=1}^k \omega(q_j)$, thus
$$ \left(\prod_{j=1}^k \omega(q_j)\right) \sum_n f(n) = \sum_n f(n) \nu(n) = \sum_{\substack{a \in \Z/q_1 \dots q_k\Z\\q_1,\dots,q_k \nmid a}} c_a \sum_n f(n) e(an/q_1 \dots q_k).$$
Applying Cauchy--Schwarz and \eqref{ca2}, one obtains the claim.
\end{proof}

\begin{corollary}[Large sieve]\label{sieve} Let ${\mathcal Q}$ be a finite collection of primes $q$.  Let $I$ be an interval of length $|I| \geq 1$, and for each $q \in {\mathcal Q}$ let us remove $\omega(q)$ residue classes modulo $q$ from $I$.  Then the number of surviving natural numbers in $I$ is at most
$$ \ll \frac{|I|}{\sum_{k=0}^\infty \sum_{\substack{q_1<\dots<q_k\\q_1,\dots,q_k \in {\mathcal Q}\\q_1 \dots q_k \leq \sqrt{|I|}}}
\prod_{j=1}^k \frac{\omega(q_j)}{q_j - \omega(q_j)}}.$$
\end{corollary}

\begin{proof}  Let $E$ denote the set of surviving natural numbers in $I$.  Then by \Cref{montgomery} with $f=1_E$ we can lower bound the quantity
\begin{equation}\label{kang}
\sum_{k=0}^\infty \sum_{\substack{q_1<\dots<q_k\\q_1,\dots,q_k \in {\mathcal Q}\\q_1 \dots q_k \leq \sqrt{|I|}}}
\sum_{\substack{a \in \Z/q_1 \dots q_k\Z\\q_1,\dots,q_k \nmid a}} |\sum_n 1_E(n) e(-an/q_1 \dots q_k)|^2
\end{equation}
by
$$
|E|^2 \sum_{k=0}^\infty \sum_{\substack{q_1<\dots<q_k\\q_1,\dots,q_k \in {\mathcal Q}\\q_1 \dots q_k \leq \sqrt{|I|}}}
\prod_{j=1}^k \frac{\omega(q_j)}{q_j - \omega(q_j)}.$$
On the other hand, the fractions $\frac{a}{q_1 \dots q_k}$ that appear in \eqref{kang} are distinct modulo $1$, and at least $1/|I|$-separated, thus by the large sieve inequality (see, e.g., \cite[Theorem 1]{montgomery-vaughan}) this expression is at most
$$ \ll |I| \sum_n |1_E(n)|^2 = |I| |E|.$$
The claim follows.
\end{proof}

Suppose $k$ is such that the product of any $k$ elements of ${\mathcal Q}$ is at most $\sqrt{|I|}$.  From symmetrization and summing the indices $q_j$ one by one, we obtain
$$
\sum_{\substack{q_1<\dots<q_k\\q_1,\dots,q_k \in {\mathcal Q}\\q_1 \dots q_k \leq \sqrt{|I|}}}
\prod_{j=1}^k \frac{\omega(q_j)}{q_j - \omega(q_j)} \geq \frac{1}{k!} \left( \sum_{q \in {\mathcal Q}}^{[-k]} \frac{\omega(q)}{q - \omega(q)}\right)^k$$
where the superscript $[-k]$ in the sum means that we remove the $k$ largest elements of the sum (with the convention that this sum vanishes if there are at most $k$ terms).  Using the trivial bound $k! \leq k^k$, we conclude the following simplified version of \Cref{sieve} (cf., \cite[Theorem 3.1]{vaughan}):

\begin{corollary}[Simplified large sieve]\label{large-sieve} Let ${\mathcal Q}$ be a finite collection of primes $q$.  Let $I$ be an interval of length $|I| \geq 1$, and for each $q \in {\mathcal Q}$ let us remove $\omega(q)$ residue classes modulo $q$ from $I$, and let $k \geq 0$ be such that any product of $k$ elements of ${\mathcal Q}$ is at most $\sqrt{|I|}$.  Then the number of surviving elements is at most
$$ \ll \frac{|I|}{\left( \frac{1}{k}\sum_{q \in {\mathcal Q}}^{[-k]} \frac{\omega(q)}{q - \omega(q)}\right)^k},$$
with the convention that the bound is vacuously true when the denominator vanishes.
\end{corollary}

\subsection{Linear relations between squares or powerful numbers}

With some elementary algebraic number theory, we have a good bound on the number of solutions to linear equations involving square numbers (or equivalently, counting lattice points on a hyperbola).

\begin{lemma}[Squares in linear relation]\label{squarecount}  Let $x$ be large, and let $a,b$ be natural numbers and $h$ a non-zero integer with $a,b,h \ll x^{O(1)}$.  Then
    $$ \# \{ (n,m) \in \N^2: an^2+h = bm^2, n \leq x \} \ll x^{o(1)}$$
    as $x \to \infty$.
\end{lemma}

\begin{proof}  This is essentially a special case of \cite[Lemma 4]{cilleruelo}.  In fact, as demonstrated in \cite[Proposition 1]{cilleruelo} (see also \cite{vaughan-wooley}), one can replace $an^2+h-bm^2$ by a more general quadratic form of height $x^{O(1)}$ (thus replacing the hyperbola by a more general conic section). For the convenience of the reader we give a self-contained proof of this result here, which was initially provided by ChatGPT and then rewritten by the author, and uses the same basic input as the arguments in \cite{cilleruelo}, namely the theory of the Pell equation.
    
Write $ab = Dc^2$ for some squarefree $1 \leq D \ll x^{O(1)}$, then the equation can be rearranged as
\begin{equation}\label{bmn}
(bm+cn\sqrt{D})(bm-cn\sqrt{D}) = bh.
\end{equation}
The right-hand side of \eqref{bmn} is of size $O(x^{O(1)})$. If $D=1$, then by the divisor bound there are only $O(x^{o(1)})$ possible choices for $bm+cn$ and $bm-cn$, giving the claim.

Now suppose that $D>1$.  The left-hand side of \eqref{bmn} can be viewed as the norm of $bm+cn\sqrt{D}$ in the ring of integers ${\mathcal O}$ of $\Q(\sqrt{D})$. It will thus suffice to show that for a given norm $R = O(x^{O(1)})$, that there are at most $O(x^{o(1)})$ elements of ${\mathcal O}$ of norm $R$ and height\footnote{The precise definition of height is not crucial for this argument, but one can for instance assign to $m+n\sqrt{D}$ the naive height $\max(|m|, |n|, D)$.}  $O(x^{O(1)})$.  By the theory of the Pell equation (see, e.g., \cite{pell}), the elements of norm $R$ split into orbits of the norm one unit group, which is a subgroup (of index at most two) of the group generated by a single fundamental unit $\eps_D$ and $\{\pm 1\}$.  The fundamental unit can be chosen to be of the form $\frac{s+t\sqrt{D}}{2}$ with $s,t \geq 1$ integers, so in particular $\eps_D \geq \frac{1+\sqrt{2}}{2} > 1$.  The heights of integer powers $\eps_D^j$ of the fundamental unit are then lower bounded by $\gg \exp(c|j|)/D^{O(1)}$ for some absolute constant $c>0$, so each orbit can only contain $O(\log x^{O(1)}) = O(x^{o(1)})$ elements of height $O(x^{O(1)})$ (cf., \cite[\S 4]{aktas}).  Thus it suffices to show that there are $O(x^{o(1)})$ orbits of norm $R$.  The number of such orbits is bounded by the number of ideal divisors of $(R)$.  Because each rational prime $(p)$ splits into at most two prime ideals in $\mathcal{O}$, the number of ideal divisors of $(R)$ is at most the square of the number of rational divisors of $R$, and the claim follows from the divisor bound.
\end{proof}

Now we replace squares by powerful numbers.

\begin{corollary}[Powerful numbers in linear relation]\label{powercount} Let $x$ be large, and let $a,b$ be natural numbers and $h$ a non-zero integer with $a, b, h \ll x$.  Then
 $$ \# \{ (n,m) \in \mathcal{VB}^1 \times \mathcal{VB}^1: an+h = bm, an \leq x \} \ll x^{\frac{2}{5}+o(1)}.$$
\end{corollary}

This result was already established in the $a=b=1$ case in \cite{chan} (and the $a=b=h=1$ case, with the $x^{o(1)}$ loss deleted, is precisely \eqref{vbvb}).  Our arguments here are broadly similar to those in \cite{aktas} to prove \eqref{vbvb}.  In view of \eqref{erd-consec} (and standard probabilistic heuristics) it is natural to conjecture that the exponent $\frac{2}{5}$ can be deleted in the above corollary, but we make no progress in this direction.

\begin{proof} Every powerful number is the product of a square and a cube, so it suffices to show that\footnote{One could additionally impose the requirement that $n_2, m_2$ be squarefree if desired, although this is unlikely to improve the bounds by more than a multiplicative constant.}
$$ \# \{ (n_1,n_2,m_1,m_2) \in \N^4: an_1^2 n_2^3+h = bm_1^2 m_2^3, a n_1^2 n_2^3 \ll x \} \ll x^{\frac{2}{5}+o(1)}.$$
Note that the constraint $a n_1^2 n_2^3 \ll x$ is equivalent to $b m_1^2 m_2^3 \ll x$.
By dyadic decomposition (paying some factors of $\log x = x^{o(1)}$) we may impose the constraints
$$ n_1 \asymp N_1; \quad n_2 \asymp N_2; \quad m_1 \asymp M_1; \quad m_2 \asymp M_2$$
for some $1 \leq N_1,N_2,M_1,M_2 \ll x$ (the bounds we give below will be uniform with respect to these parameters). The contribution vanishes unless
\begin{equation}\label{nmnm}
 N_1^2 N_2^3, M_1^2 M_2^3 \ll x.
\end{equation}
We estimate the count in several different ways.  There are $O(N_1 N_2)$ choices for $n_1, n_2$, and then by the divisor bound there are at most $O(x^{o(1)})$ choices for $m_1,m_2$.  Thus one upper bound is
$$ \ll N_1 N_2 x^{o(1)};$$
similarly we have an upper bound of
$$ \ll M_1 M_2 x^{o(1)}.$$
Next, there are $O(N_2 M_2)$ choices for $n_2, m_2$, and then by \Cref{squarecount} there are at most $O(x^{o(1)})$ choices for $n_1, m_1$, so another bound is
$$ \ll N_2 M_2 x^{o(1)};$$
Raising the first two bounds to the power $2/5$ and the third to the power $1/5$, we obtain an averaged bound of
$$ \ll (N_1^2 N_2^3)^{1/5} (M_1^2 M_2^3)^{1/5} x^{o(1)}$$
and the claim follows from \eqref{nmnm}.
\end{proof}

\begin{remark}\label{rem-divisor}  Using a more refined version of the divisor bound, one can improve the $x^{o(1)}$ factors in the above bounds slightly to $x^{O(1/\log_2 x)}$.  The proof of \Cref{powercount} also reveals that for the purposes of improving the $\frac{2}{5}$ exponent, the critical case is when $n_1, n_2, m_1, m_2$ are all comparable to $x^{1/5}$.  A model problem here is to show that the number of integer points on the three-dimensional variety
\begin{equation}\label{nm}
\{ (n_1,n_2,m_1,m_2): n_1^2 n_2^3+1 = m_1^2 m_2^3\}
\end{equation}
of height at most $H$ is $o(H^2)$ as $H \to \infty$; by using the arguments from \cite{aktas}, this would allow one to improve the right-hand side of \eqref{vbvb} to $o(x^{2/5})$ as $x \to \infty$.  See \cite{tao-elsholtz}, \cite{hb-cubic} for some other problems on counting integer points on varieties that have a somewhat similar flavor.  In\footnote{We thank Tim Browning for these references.} \cite{browning-vanvalckenborgh}, the question of counting the triples $a,b,c$ of powerful numbers with $a+b=c$ was considered.  The corresponding model problem \eqref{nm} in this context is that of counting integer points in the five-dimensional variety
$$\{ (n_1,n_2,m_1,m_2, l_1, l_2): n_1^2 n_2^3+ m_1^2 m_2^3 = l_1^2 l_2^3\}.$$
of height at most $H$.  The arguments in \cite{browning-vanvalckenborgh} give an upper bound of $O(H^{3+o(1)})$, which is analogous to the upper bound of $O(H^{2+o(1)})$ on \eqref{nm} that can be obtained by our methods.  This was very recently improved to $O(H^{3-\frac{3}{311}+o(1)})$ by Heath-Brown \cite{hb-counting}.  The arguments rely on extracting additional linear congruence relations between $n_1, n_2, m_1, m_2$ modulo $l_2$, $l_2^2$, and $l_2^3$, after fixing the ratio mod $l_2$ of $m_1 m_2$ and $l_1 l_2$; see \cite[Lemma 4.1]{hb-counting}.  It is not clear to the author whether similar arguments are available in the lower-dimensional setting of counting points on \eqref{nm}.
\end{remark}

\section{Controlling the very bad intervals}\label{verybad-sec}

In this section we establish \Cref{verybad-thm}.  We begin by using the equidistribution estimate (\Cref{vinogradov}) to show that very bad intervals cannot be too long.

\begin{lemma}[Very bad intervals are short]\label{hvin}  Let $\{N+1,\dots,N+H\}$ be a very bad interval.  Then $H < N$.  Furthermore one has the improved bound
$$ H \leq \exp(\log^{2/3+o(1)} N)$$
as $N \to \infty$.
\end{lemma}

\begin{proof} If $H \geq N$ then by \Cref{prime-interval}(i) there is a prime $(N+H)/2 < p \leq N+H$ that divides the product \eqref{nhint} exactly once, a contradiction.  This gives the first claim $H < N$.   One could improve this bound using \Cref{prime-interval}(ii), but we can obtain a further improvement as follows.  If $p_0$ is the largest prime factor of \eqref{nhint}, then by \Cref{classical}(i) we have $p_0 > H$, so $p_0$ divides exactly one of the $N+1,\dots,N+H$, and then must divide it at least twice by \Cref{def-bad}(ii).  Thus
$$H < p_0 \leq \sqrt{N+H} \leq \sqrt{2N}.$$

It remains to show that
$$ H \leq \exp(\log^{2/3+\eps} N)$$
whenever $\eps>0$ and $N$ is sufficiently large depending on $\eps$.  Suppose for contradiction that
\begin{equation}\label{hnl}
 H > \exp(\log^{2/3+\eps} N).
\end{equation}

Every prime $p$ in the range $H < p \leq 2H$ divides at most one of the $N+1,\dots,N+H$, and if it does, must divide it at least twice.  This places a constraint on the fractional parts of $N/p$ and $N/p^2$. For instance, one cannot simultaneously have
    $$ 0.9 \leq \left\{\frac{N}{p}\right\} < 1$$
    and
    $$ 0 \leq \left\{\frac{N}{p^2}\right\} < 0.9$$
    as this would imply the existence of $1 \leq h \leq H$ such that $N+h$ is divisible by $p$ but not $p^2$.
    To exploit this, we introduce a smooth $\Z^2$-periodic function $W \colon \R^2 \to [0,1]$ supported on the region
    $$ \left\{ (t_1,t_2): 0.9 \leq \{t_1\} < 1; 0 \leq \{ t_2\} < 0.9 \right\}$$
    that equals one on
    $$ \left\{ (t_1,t_2): 0.91 \leq \{t_1\} \leq 0.99; 0.01 \leq \{ t_2\} < 0.89 \right\}$$
    and has a $C^3$ norm of $O(1)$.  By the previous discussion, we have
    $$ \sum_{H < p < 2H} W\left(\frac{N}{p}, \frac{N}{p^2}\right) = 0$$
    and hence by \Cref{vinogradov}
$$ \int_H^{2H} W\left(\frac{N}{t}, \frac{N}{t^2}\right) \frac{dt}{\log t} \ll \frac{H}{\log^{10} H}$$
(say), or on making the change of variables $s \coloneqq \frac{N}{t}$
    $$ \int_{N/2H}^{N/H} W\left(s, \frac{s^2}{N}\right) \ ds \ll \frac{N}{H\log^9 H}.$$
By the construction of $W$, we thus have
$$  \left| \left\{ \frac{N}{2H} \leq s \leq \frac{N}{H}: 0.91 \leq \left\{s\right\} \leq 0.99; 0.01 \leq \left\{ \frac{s^2}{N}\right\} \leq 0.89 \right\}\right| \ll \frac{N}{H\log^9 H}.$$
On any unit subinterval of $[\frac{N}{2H},\frac{N}{H}]$, the quantity $\frac{s^2}{N}$ only varies by $O(1/H)$, which is $o(1)$ thanks to \eqref{hnl}.  Hence, if one has
$$ 0.02 \leq \left\{ \frac{s^2}{N}\right\} \leq 0.88$$
for any $s$ in such a subinterval, then that subinterval will contribute at least $0.08$ to the above measure.  Covering $[\frac{N}{2H},\frac{N}{H}]$ by $\asymp \frac{N}{H}$ such intervals with bounded overlap, we conclude that
$$ \left| \left\{ \frac{N}{2H} \leq s \leq \frac{N}{H}: 0.02 \leq \left\{ \frac{s^2}{N}\right\} \leq 0.88 \right\}\right| \ll \frac{N}{H\log^9 H}$$
which after applying the further change of variables $u \coloneqq \frac{s^2}{N}$ becomes
\begin{equation}\label{valid}
 \left| \left\{ \frac{N}{4H^2} \leq u \leq \frac{N}{H^2}: 0.02 \leq \{ u \} \leq 0.88 \right\}\right|
\ll \frac{N}{H^2 \log^9 H}.
\end{equation}
Because $H \leq \sqrt{2N}$, we have $L \coloneqq \frac{N}{H^2} \geq \frac{1}{2}$.  We claim that the left-hand side of \eqref{valid} is $\gg L$; see also \Cref{fig-slice}.  Indeed, if $L \geq \frac{4}{3}$, then the interval $[\frac{L}{4}, L]$ has length $\frac{3}{4} L \geq 1$, and any interval of length $\ell \geq 1$ meets $\{ u: 0.02 \leq \{u\} \leq 0.88\}$ in a set of measure at least $0.86 \lfloor \ell \rfloor \gg \ell$; while if $\frac{1}{2} \leq L < \frac{4}{3}$, then $[\frac{L}{4}, L]$ contains $[\frac{1}{3}, \frac{1}{2}]$, which lies in this set and has measure $\frac{1}{6} \gg L$.  This gives the required contradiction.
\end{proof}

\begin{figure}
    \centering
\includegraphics[width=0.75\textwidth]{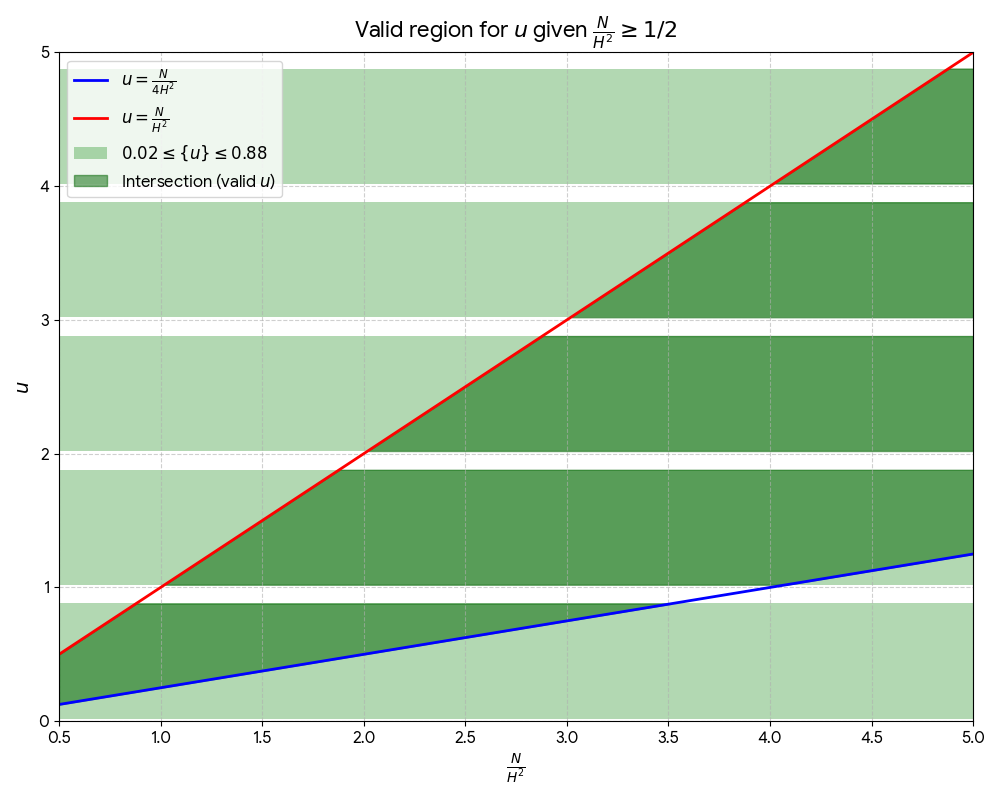}
    \caption{For a given choice of $\frac{N}{H^2}$, the set in \eqref{valid} is the shaded region between the blue and red lines intersected with the vertical line at $\frac{N}{H^2}$.  For $\frac{N}{H^2} \geq \frac{1}{2}$, this measure is always comparable to the distance $\frac{3}{4} \frac{N}{H^2}$ between these lines. (Image generated by Gemini.)}
    \label{fig-slice}
\end{figure}

Next, we show that very bad intervals of non-trivial length create a linear relation between two powerful numbers.

\begin{lemma}[Very bad intervals create a linear relation]\label{linrel}  Let $\{N+1,\dots,N+H\}$ be a very bad interval of length $H > 1$.  Then there is a solution to the equation
    $$ an + h = bm$$
    with $an, bm \in \{N+1,\dots,N+H\}$, $1 \leq h \leq H$, $a,b \ll H^{O(1)}$, and $n,m$ powerful.
\end{lemma}

\begin{proof}  Every prime $p>H$ divides at most one of the $N+1,\dots,N+H$, and if it does, it must divide at least twice, since $\{N+1,\dots,N+H\}$ is very bad.  Thus every element $N+h$ of this interval factors as $N+h = a_h n_h$, where $n_h$ is powerful and $a_h$ is the product of some distinct primes less than or equal to $H$.  Each prime $p \leq H$ can only divide at most $H/p+1$ of the $a_h$, thus
    $$ \prod_{h=1}^H a_h \leq \prod_{p \leq H} p^{H/p + 1} = \exp\left( H \sum_{p \leq H} \frac{\log p}{p} + \sum_{p \leq H} \log p\right) \ll H^{O(H)}$$
by the prime number theorem.  Since $H > 1$, we conclude from Markov's inequality (after taking logarithms) that we can find $1 \leq h_1 < h_2 \leq H$ such that $a_{h_1}, a_{h_2} \ll H^{O(1)}$.  The claim follows (by setting $a$, $b$, $h$, $n$, $m$ to be $a_{h_1}$, $a_{h_2}$, $h_2-h_1$, $n_{h_1}$, $n_{h_2}$ respectively).
\end{proof}

We can now prove \Cref{verybad-thm}.  By dyadic decomposition, it suffices to show that the union of the very bad intervals $\{N+1,\dots,N+H\}$ with $x/2 < N \leq x$ and $H > 1$ has cardinality
\begin{equation}\label{good}
    \ll x^{\frac{2}{5}+o(1)}.
\end{equation}
From \Cref{hvin} and \Cref{linrel}, each very bad interval is associated to a solution to the equation $an + h = bm$ with $a,b,h = x^{o(1)}$ and $n,m$ powerful with $an \ll x$.  By \Cref{powercount} (summing over the $x^{o(1)}$ possible choices of $(a,b,h)$), there are only $x^{\frac{2}{5}+o(1)}$ such solutions; and each pair contributes to at most $O(x^{o(1)})$ very bad intervals, each of which has length $O(x^{o(1)})$ by a further appeal to \Cref{hvin}. The claim follows.

\section{Controlling the \texorpdfstring{$F_3$}{F3} intervals}\label{f3-sec}

We now adapt the arguments of the previous section to prove \Cref{f3-thm}.  We first need a bound on the quantity $a$ appearing in \Cref{def-bad}(iii), which is implicit in \cite[p. 343]{erdos-graham} (and also appears in \cite[(10)]{luca}):

\begin{lemma}[Bound on $a$]\label{abound}  Let $\{N+1,\dots,N+H\}$ be an $F_3$ interval, so that the product \eqref{nhint} has the same squarefree component as $a!$ for some $1 \leq a < N$.  Then $H < N$ and $a \ll H \log N$.
\end{lemma}

\begin{proof} If $H \geq N$, then from \Cref{prime-interval}(i) there is a prime $(N+H)/2 < p \leq N+H$ that divides the product \eqref{nhint} exactly once and does not divide $a!$, contradicting the definition of an $F_3$ interval.  Thus we have $H < N$.

Every prime in $(a/2,a]$ divides $a!$ exactly once, and thus must divide $(N+1) \dots (N+H)$.  Taking logarithms, we conclude that
$$ \sum_{a/2 < p \leq a} \log p \leq \sum_{h=1}^H \log (N+h) \ll H \log N$$
and the claim follows from the prime number theorem.
\end{proof}

Next, we establish an analogue of \Cref{hvin}:

\begin{lemma}[$F_3$ intervals are short]\label{hf3}  Let $\{N+1,\dots,N+H\}$ be a type $F_3$ interval.  Then
\begin{equation}\label{HBound}
 H \leq \exp(\log^{2/3+o(1)} N)
\end{equation}
as $N \to \infty$.
\end{lemma}

\begin{proof}  Let $\eps > 0$.  It suffices to show that
$$ H \leq \exp(\log^{2/3+\eps} N)$$
for sufficiently large $N$ depending on $\eps$.  Suppose for contradiction that
$$ H > \exp(\log^{2/3+\eps} N).$$
Since $H \leq N$, we see from \Cref{def-bad}(iii) that the interval $\{N+1,\dots,N+H\}$ cannot contain a prime, and thus by \Cref{prime-interval}(ii) one has
\begin{equation}\label{A3}
        H \ll N^{0.525}.
\end{equation}

Let $P \coloneqq H \log^2 N$, then by \Cref{abound} every prime $p$ in the range $P < p \leq 2P$ does not divide $a!$, and thus divides $(N+1) \dots (N+H)$ an even number of times.  We have $p > P > H$, thus $p$ can only divide one of the $N+1,\dots,N+H$, and if so it is divisible by $p^2$.
 If $P \leq \sqrt{2N}$, then we may argue similarly to the proof of \Cref{hvin} as follows.  For each $P < p \leq 2P$, one cannot simultaneously have
 $$1 - \frac{1}{10 \log^2 N} \leq \frac{N}{p} < 1$$
 and
 $$0 \leq \frac{N}{p^2} < 0.9.$$
We can then introduce a smooth $\Z^2$-periodic function $W \colon \R^2 \to \R$ supported on the set
$$ \left\{ (u,v) \in \R^2 : 1 - \frac{1}{10 \log^2 N} \leq u < 1, \ 0 \leq v < 0.9 \right\}$$
that equals $1$ on the set
$$ \left\{ (u,v) \in \R^2 : 1 - \frac{1}{20 \log^2 N} \leq u < 1 - \frac{1}{40 \log^2 N}, \ 0 \leq v < 0.8 \right\}$$
with a $C^3$ norm of $O(\log^6 N)$.  One can then repeat the argument from the proof of \Cref{hvin} (with some acceptable additional losses of $\log^6 N$) to obtain the required contradiction.
 
 This concludes the treatement of the case $P \leq \sqrt{2N}$.  By \eqref{A3}, we may now assume that $P$ lies in the range
 \begin{equation}\label{P-big}
    \sqrt{2N} < P \ll N^{0.525} \log^2 N.
 \end{equation}
Now none of the $N+1, \dots, N+H$ can be divisible by $p^2$ for $P < p \leq 2P$, and so cannot be divisible by $p$ either.  Thus if we let $W \colon \R \to \R$ be a $1$-periodic function supported on the set
$$ \left\{ u \in \R : 1-\frac{1}{10\log^2 N} \leq \left\{ u \right\} < 1 \right\}$$
which equals $1$ on
$$ \left\{ u \in \R : 1-\frac{1}{20\log^2 N} \leq \left\{ u \right\} < 1- \frac{1}{30 \log^2 N} \right\}$$
and has a $C^3$ norm of $O(\log^6 N)$, then
$$ \sum_{P \leq p < 2P} W\left(\frac{N}{p}\right) = 0.$$
Applying \Cref{vinogradov} with weight function $(t_1,t_2) \mapsto W(t_1)$ (and \eqref{P-big}) we conclude
$$ \int_P^{2P} W\left(\frac{N}{t}\right) \frac{dt}{\log t} \ll \frac{P}{\log^{10} N}$$
(say).  Applying the change of variables $u = \frac{N}{t}$, we conclude
$$  \int_{N/2P}^{N/P} W\left(u\right) du \ll \frac{N}{P \log^{9} N}.$$
From \eqref{P-big} we have $N/P \geq 10$ (say).  Since $W$ has integral at least $\frac{1}{60 \log^2 N}$ on every unit interval, we have
$$  \int_{N/2P}^{N/P} W\left(u\right) du \gg \frac{N}{P \log^2 N},$$
giving the required contradiction.
\end{proof}

Now we establish an analogue of \Cref{linrel}:

\begin{lemma}[$F_3$ intervals create a linear relation]\label{linrel-f3}  Let $\{N+1,\dots,N+H\}$ be an $F_3$ interval of length $H > 1$.  Then there is a solution to the equation
    $$ a_1 n_1^2 + h = a_2 n_2^2$$
    with $a_1 n_1^2, a_2 n_2^2 \in \{ N+1, \dots, N+H\}$, $1 \leq h \leq H$, and $1 \leq a_1, a_2 \ll e^{O(P/H)} H^{O(1)}$, and $a_1,a_2$ $P$-smooth, where $P \coloneqq \max(a,H)$.
\end{lemma}

\begin{proof}  Every prime larger than $P$ is coprime to $a!$ and divides at most one of the $N+1,\dots,N+H$, and so must divide each such number an even number of times by \Cref{def-bad}(iii).  Thus we can write $N+h = a_h n_h^2$ where $a_h$ is the product of distinct primes up to $P$, in particular the $a_h$ are $P$-smooth and hence $O(H \log N)$-smooth by \Cref{abound}. Arguing as in the proof of \Cref{linrel}, we conclude that
$$ \prod_{h=1}^H a_h \leq \prod_{p \leq H} p^{H/p} \cdot \prod_{p \leq P} p = \exp(O( H \log H + P ))$$
by the prime number theorem. Thus by Markov's inequality we can find $1 \leq h_1 < h_2 \leq H$ such that $a_{h_1}, a_{h_2} \ll \exp(O(\log H + P/H))$, giving the claim (by setting $a_1$, $a_2$, $h$, $n_1$, $n_2$ to be $a_{h_1}$, $a_{h_2}$, $h_2-h_1$, $n_{h_1}$, $n_{h_2}$ respectively).
\end{proof}

Now suppose that $\{N+1,\dots,N+H\}$ is an $F_3$ interval of length $H > 1$ for which $x/2 < N+H \leq x$, which by \Cref{abound} and \Cref{hf3} implies that $N \asymp x$ and $a, H = x^{o(1)}$.  It will suffice to show that the number of possible choices for $N+H$ is $O(x^{1/2+o(1)})$.  By paying factors of $x^{o(1)}$, we may assume that $a$ and $H$ are fixed.

We split into several cases.  First consider the case where either $a \leq \eps H \log N$ for some sufficiently small absolute constant $\eps > 0$, or $H = O(1)$.  Let $a_1, a_2, h, n_1, n_2$ be as in the above lemma, then either $a_1,a_2 = x^{O(\eps)}$, or else $a_1,a_2$ are $O(\log x)$-smooth.  In either case there are $O(x^{O(\eps)})$ choices for $a_1,a_2, h$ (where in the latter case we use \Cref{smoot}(ii) with $A=1+\eps$, say), and once one fixes these parameters, there are $O(x^{o(1)})$ choices for $n_1, n_2$ thanks to \Cref{squarecount} (applied with $x$ replaced by $x^{1/2}$).  The quantity $N+H$ lies within $H = x^{o(1)}$ of $a_1 n_1^2$, hence each choice of $a_1, n_1$ generates at most $x^{o(1)}$ possibilities for $N+H$.  Thus the total number of $N+H$ of this form is $O( x^{O(\eps) + o(1)} )$, which is acceptable if $\eps$ is small enough.  Thus, by \Cref{abound}, we may now restrict to the case
\begin{equation}\label{ahn}
    a \asymp H \log N \asymp H \log x
\end{equation}
and where $H$ is larger than any given absolute constant; this is the regime where our estimates are weakest.

For every prime $p$ with $a/2 < p \leq a$, $p$ divides $a!$ exactly once, and hence must divide one of $N+1,\dots,N+H$.  In particular, this restricts $N$ to at most $H$ residue classes modulo $p$.  Applying \Cref{large-sieve},
we see that the number of choices for $N$ (or $N+H$) is at most
$$ \ll \frac{x}{\left( \frac{1}{k} \sum_{a/2 < p \leq a}^{[-k]} \frac{p-H}{H}\right)^k}$$
whenever $a^k \leq \sqrt{x}$.  Using the prime number theorem and \eqref{ahn}, we thus obtain the bound
\begin{equation}\label{cak}
 \ll \frac{x}{\left( c \frac{a}{\log a} \frac{\log x}{k} \right)^k}
\end{equation}
provided that
\begin{equation}\label{kca}
k \leq c \frac{a}{\log a}
\end{equation}
and $c>0$ is a sufficiently small constant.  We choose $k$ to be the largest integer such that $a^k \leq \sqrt{x}$, then $k \log a \asymp \log x$ and so the condition \eqref{kca} follows from \eqref{ahn} and the fact that $H$ is large.  Since $\log x \ll a = x^{o(1)}$, we have $a^k = x^{1/2-o(1)}$ and $k \ll \frac{\log x}{\log_2 x} = o(\log x)$, and the bound \eqref{cak} then simplifies to $x^{1/2+o(1)}$, giving the claim.  This proves \Cref{f3-thm}.

We can now prove \Cref{thm-f3-eq}.  From \eqref{f31} we know that the number of solutions is $\gg x^{1/2}$; it thus remains to establish the upper bound $\ll x^{1/2+o(1)}$.
    If $a_1,a_2,a_3,m$ solve \eqref{f3-eq} with $a_1 < a_2 < a_3 \leq x$, then $\{a_2+1,\dots,a_3\}$ is a type $F_3$ interval by \Cref{def-bad}(iii), and hence by \Cref{f3-thm} there are $O(x^{1/2+o(1)})$ choices for $a_3$.  From \Cref{hf3} we have $a_3 - a_2 \ll x^{o(1)}$, so once $a_3$ is fixed there are only $x^{o(1)}$ choices for $a_2$.  Once $a_2, a_3$ are fixed, there are at most two choices for $a_1$, since by \Cref{classical}(ii) the squarefree part $s(a_1!)$ of $a_1!$ only repeats when $a_1$ is a perfect square.  Since $m$ is determined by $a_1,a_2,a_3$, the claim follows.

\begin{remark}\label{rem-cover} The same argument shows that the set of elements of $[1,x]$ that lie in at least one $F_3$ interval has cardinality $x^{1/2+o(1)}$.
\end{remark}

\section{Character sums over primes}

Given a Dirichlet character $\chi$ and a scale $Z \geq 2$, let $s_Z(\chi)$ denote the normalized character sum over primes
\begin{equation}\label{sz-def}
 s_Z(\chi) \coloneqq \frac{\sum_{Z \leq p < 2Z} \chi(p)}{\sum_{Z \leq p < 2Z} 1}.
\end{equation}
Typically we will be interested in the regime where the conductor of $\chi$ is of polynomial size $O(Z^{O(1)})$.  Trivially one has
$$ |s_Z(\chi)| \leq 1,$$
but we will be interested in obtaining a power savings in this bound, similar to \Cref{burgess-bound}.  To this end, call a non-principal Dirichlet character $\chi$ \emph{exceptional} (with respect to the scale $Z$) if
\begin{equation}\label{excep}
    |s_Z(\chi)| \geq Z^{-0.008},
\end{equation}
and \emph{unexceptional} otherwise.
We chose to use explicit numerical exponents here in order to reduce the number of small unspecified parameters such as $\eps$ in our arguments.  Roughly speaking, exceptional characters of conductor dividing $q$ are the main obstruction to random products of primes, such as $\mathbf{p}_0^2 \mathbf{p}_1 \dots \mathbf{p}_{1000}$, equidistributing with respect to the primitive residue classes of a given modulus $q$.

Due to the possible presence of Siegel zeroes, one cannot expect to eliminate exceptional characters entirely; however, we can obtain good upper bounds on the number of such characters:

\begin{lemma}[Few exceptional characters]\label{primesum}  Let $Z$ be a large quantity, and let $q_1 \leq Z^{3.09}$ be cubefree.  Let $\chi_1,\dots,\chi_J$ be exceptional primitive characters, with each $\chi_j$ of period $q_1 q_{2,j}$ for some cubefree $q_{2,j} \leq (Z^{3.09}/q_1)^{1/2}$ coprime to $q_1$.  Then
\begin{equation}\label{jp} \sum_{j=1}^J |s_Z(\chi_j)|^2 \ll 1.
\end{equation}
In particular, from \eqref{excep} we have
\begin{equation}\label{J-bound} J \ll Z^{0.016}. \end{equation}
More generally, for any $Z^{-0.008} \leq \lambda \leq 1$, one has $|s_Z(\chi_j)| \leq \lambda$ for all but $O(\lambda^{-2})$ of the $\chi_j$.
\end{lemma}

In order to prove this lemma, we recall some standard facts.  We begin with a version of the Burgess bound (with explicit numerical exponents):

\begin{lemma}[Burgess bound]\label{burgess-bound}  Let $Y$ be a large quantity, and let $q$ be cube-free and obey the bound $q \ll Y^{3.1}$.  Then for any non-principal Dirichlet character $\chi$ of period $q$, one has
$$ \sum_{n \leq Y} \chi(n) \ll Y^{1-0.0163}.$$
\end{lemma}

As is well known, the Burgess bounds offer power savings for $q$ as large as $Y^{4-\eps}$, but for our purposes any exponent larger than $3$ will suffice.

\begin{proof}  Applying the main results of \cite{burg1}, \cite{burg2}, \cite{burg3}, we have
$$ \sum_{n \leq Y} \chi(n) \ll_{r,\eps'} Y^{1-1/r} q^{(r+1)/4r^2 + \eps'}$$
for any $r,\eps'>0$.  Taking $r=7$ and $\eps'$ sufficiently small, we obtain the claim after a brief numerical calculation.
\end{proof}

Next, we recall the Bombieri--Hal\'asz--Montgomery inequality \cite{bombieri}:

\begin{lemma}[Bombieri--Hal\'asz--Montgomery inequality]\label{bombieri-halasz-montgomery} Let $\xi, \phi_1,\dots,\phi_J$ be vectors in a complex Hilbert space.  Then
$$ \sum_{j=1}^J|\langle \xi, \phi_j \rangle|^2 \leq \| \xi \|^2 \sup_{1 \leq j \leq J} \sum_{j'=1}^J |\langle \phi_j, \phi_{j'}\rangle|.
$$
\end{lemma}

We refer the reader to \cite{harcos} for some short proofs of this inequality.

Finally, we recall a version of the fundamental lemma of sieve theory:

\begin{lemma}[Fundamental lemma of sieve theory]\label{fund-lemma}  Let $R > 1$ be large.  Then there exist weights $\lambda_d \in \{-1,0,1\}$ that are only non-vanishing when $d$ is a squarefree number less than or equal to $R$, with the sieve
$$ \nu(n) \coloneqq \sum_{d|n} \lambda_d$$
non-negative, with $\lambda_1 = 1$ and
$$ \sum_{d} \frac{\lambda_d}{d} \ll \frac{1}{\log R}.$$
In particular, this implies that
\begin{equation}\label{nx}
     \sum_{n \leq X} \nu(n) = \sum_d \lambda_d \left(\frac{X}{d}+O(1)\right) \ll \frac{X}{\log R}
\end{equation}
for any $X \gg R^2$ (say).
\end{lemma}

\begin{proof} See for instance \cite[Lemma 6.3]{ik}.  One can optimize the exponents and implicit constants further, but we will not need to do so here.
\end{proof}

We are now ready to prove \Cref{primesum}.  We will prove the lemma under the additional hypothesis $J \leq C Z^{0.016}$ for some large constant $C$ (and with $Z$ assumed sufficiently large depending on $C$), with the implied constant in \eqref{jp} independent of $C$. By \eqref{J-bound}, this will mean that $J$ cannot equal $\lfloor C Z^{0.016} \rfloor$ if $C$ is large enough, and hence $J$ cannot exceed $C Z^{0.016}$ either, giving the claim.

In order to obtain a ``log-free'' estimate, we use the standard device of inserting a sieve.
Let $\nu$ and $\lambda$ be as in \Cref{fund-lemma} with $R \coloneqq Z^{0.0001}$.  We apply \Cref{bombieri-halasz-montgomery} with the Hilbert space\footnote{Technically, this is a pre-Hilbert space as some vectors can have norm zero, but one can easily quotient out such vectors to recover a Hilbert space.} of sequences $f \colon \N \to \C$ with the inner product
$$ \langle f, g \rangle \coloneqq \sum_{n < 2Z} f(n) \overline{g(n)} \nu(n)$$
with the vector $\xi$ defined by setting $\xi(p)=1$ for primes $Z \leq p < 2Z$ and $\xi(n)=0$ for all other $n$, and $\phi_j = \chi_j$.
Then, as $\nu(p)=1$ for all primes $Z \leq p < 2Z$, we have
$$ |\langle \xi, \phi_j \rangle| = \left|\sum_{Z \leq p < 2Z} \chi_j(p)\right| \asymp \frac{Z}{\log Z} |s_Z(\chi_j)|,$$
and
$$ \|\xi\|^2 = \sum_{Z \leq p < 2Z} 1 \asymp \frac{Z}{\log Z}.$$
By \Cref{bombieri-halasz-montgomery}, it thus suffices to establish the bound
$$ \sum_{j'=1}^J |\langle \phi_j, \phi_{j'}\rangle| \ll \frac{Z}{\log Z}.$$
The left-hand side expands as
$$ \sum_{j'=1}^J \left|\sum_{n < 2Z} \chi_j \overline{\chi_{j'}}(n) \nu(n)\right|.$$
The diagonal contribution $j'=j$ is acceptable from \eqref{nx}.  For $j' \neq j$, we expand $\nu$ to bound this contribution by
$$ \sum_{1 \leq j' \leq J: j' \neq j} \sum_{d \leq Z^{0.0001}} \left|\sum_{n \leq 2Z/d} \chi_j \overline{\chi_{j'}}(n)\right|.$$
By hypothesis, $\chi_j \overline{\chi_{j'}}$ is a non-principal character whose period is cubefree and bounded by $q_1 q_{2,j} q_{2,j'} \leq Z^{3.09}$.  By \Cref{burgess-bound} we thus have
$$\sum_{n \leq 2Z/d} \chi_j \overline{\chi_{j'}}(n) \ll Z^{1-0.0163}$$
and the claim now follows since $J \leq C Z^{0.016}$ and $Z$ is large compared to $C$.

\section{Controlling the bad intervals}\label{main-sec}

In this section we prove \Cref{bad-thm}. Let $x$ be large.  First observe that if $\{N+1,\dots,N+H\}$ is bad, then one must have $H \leq N$, for if $H > N$, then the largest prime $p$ less than or equal to $N+H$ is the largest prime dividing \eqref{nhint}, but divides it only once, contradicting \Cref{def-bad}(i).  As a consequence, if a bad interval intersects $[x/2,x]$, then it must lie in $[x/4,2x]$.

Call an interval $\{N+1,\dots,N+H\}$ \emph{admissible} if $H > 1$ and $\{N+1,\dots,N+H\}$ is bad and is contained in $[x/4,2x]$.  By dyadic decomposition and the previous observation, it suffices to show that the union of all the admissible intervals has cardinality
\begin{equation}\label{bond}
 \ll \frac{\#({\mathcal B}^1 \cap [1,x])}{\log^{1-o(1)} x} =  \frac{\#({\mathcal B}^1 \cap [1,x])}{\log^{2-o(1)} z}
\end{equation}
as $x \to \infty$.

\subsection{Basic estimates}

We first establish some basic properties of admissible intervals.

\begin{lemma}[Basic estimates]\label{lem-basic}  Let $\{N+1,\dots,N+H\}$ be an admissible interval. Then $H \leq N$ and $N \asymp x$.  Furthermore, the interval $\{N+1,\dots,N+H\}$ contains an element of the form $p_0^2 m$ with $p_0$ a prime of size
$$ H < p_0 \ll \sqrt{x}$$
(so in particular $H \ll \sqrt{x}$) and $m$ is a $p_0$-smooth number with $m \ll x / p_0^2$.  Furthermore, all elements of $\{N+1,\dots,N+H\}$ are $p_0$-smooth, and none of them are prime.
\end{lemma}

\begin{proof} We have already explained why $H \leq N$. Since $\{N+1,\dots,N+H\}$ is contained in $[x/4,2x]$, this gives $N \asymp x$.
Let $p_0$ be the largest prime dividing \eqref{nhint}.  By \Cref{classical}(i) we have $p_0 > H$, hence $p_0$ divides only one of the $N+1,\dots,N+H$, and by \Cref{def-bad}(i) it divides this number at least twice, thus making it of the form $p_0^2 m$.  This gives the bounds $m \ll x/p_0^2$ and $p_0 \ll \sqrt{x}$.  As $p_0$ is the largest prime dividing \eqref{nhint}, all elements of $\{N+1,\dots,N+H\}$ are $p_0$-smooth.  Finally, if some $N+h$ were a prime $q$, then any prime $p > q$ dividing \eqref{nhint} would divide some $N+h' \neq q$, forcing $N+h' \geq 2p > 2N \geq N+H$, which is absurd; thus $q$ would be the largest prime dividing \eqref{nhint}, but would only divide it once (as $q > N \geq H$), contradicting \Cref{def-bad}(i).
\end{proof}

\begin{remark} The upper bound $H \ll \sqrt{x}$ was improved to $H \ll x^{1/2-c}$ for some absolute constant $c>0$ in \cite{ramachandra} using the Selberg sieve. In \cite[p. 73]{eg} it is conjectured that the upper bound on $H$ can be improved to $H \ll x^{o(1)}$, and furthermore that examples of bad intervals exist with arbitrarily large $H$, including ones whose initial value $N+1$ is the square of a prime. We do not make progress on these conjectures here; in particular, the arguments used to prove \Cref{hvin} or \Cref{hf3} do not seem to be applicable for the study of bad intervals, due to large part to a lack of strong upper bounds for $H$ or $p_0$.
\end{remark}

\subsection{Reduction to normalized intervals}

Define a \emph{normalized interval} to be an admissible interval $\{N+1,\dots,N+H\}$ for which $H$ is a power of two, and the special element $p_0^2 m$ is one of the two endpoints $N+1, N+H$ of the interval, so that the interval takes the form
\begin{equation}\label{first}
 \{ p_0^2 m, p_0^2 m + 1, \dots, p_0^2 m + H - 1 \}
\end{equation}
or
\begin{equation}\label{second}
 \{ p_0^2 m - H + 1, \dots, p_0^2 m - 1, p_0^2 m \}.
\end{equation}

We make a simple observation:

\begin{lemma}[Admissible intervals contain large normalized intervals]\label{lem-norm}  Every admissible interval $\{N+1,\dots,N+H\}$ contains a normalized interval of length $\asymp H$.
\end{lemma}

\begin{proof}
Take $H'$ to be the largest power of two less than $\frac{H+3}{2}$, then $2 \leq H' \asymp H$.  If neither of the intervals
\begin{equation}\label{first'}
 \{ p_0^2 m, p_0^2 m + 1, \dots, p_0^2 m + H' - 1 \}
\end{equation}
or
\begin{equation}\label{second'}
 \{ p_0^2 m - H' + 1, \dots, p_0^2 m - 1, p_0^2 m \}
\end{equation}
are contained in $\{N+1,\dots,N+H\}$, then $\{N+1,\dots,N+H\}$ would be contained in $\{p_0^2 m - H'+2,\dots,p_0^2 m + H' - 2\}$, which has length $2H'-3 < H$, a contradiction.  Thus at least one of the intervals \eqref{first'}, \eqref{second'} is contained in $\{N+1,\dots,N+H\}$, which easily implies that it is admissible, and therefore normalized by construction.
\end{proof}

By \Cref{lem-norm}, we obtain the pointwise bound
$$ 1_A(n) \ll M 1_N(n)$$
for all integers $n$, where $A$ and $N$ are the union of all admissible intervals and the union of all normalized admissible intervals, respectively, and $M$ is the Hardy--Littlewood maximal operator on the integers.  By the Hardy--Littlewood maximal inequality, we conclude that the union of all admissible intervals has cardinality comparable to the union of all normalized admissible intervals.  Thus, for the purposes of establishing the bound \eqref{bond}, we may without loss of generality restrict attention to normalized intervals.

For sake of notation we shall only consider normalized intervals of the form \eqref{first}, as the other case \eqref{second} is treated similarly.

\subsection{Removing non-typical intervals}

We would like to restrict further to ``typical'' normalized intervals, which we now pause to define.

\begin{definition}[Typical normalized interval]\label{def-typical}  A normalized interval $\{N+1,\dots,N+H\}$ is said to be \emph{typical} if the following properties hold:
\begin{itemize}
    \item[(i)] One has
\begin{equation}\label{hup}
    H < \log^{20} x.
\end{equation}
    \item[(ii)] None of $N+1,\dots,N+H$ is a multiple of a square $d^2$ with $d \geq z^3$.
    \item[(iii)] One has $p_0 = z^{1+o(1)}$ (for a suitably slowly decaying choice of $o(1)$).  Furthermore there exist primes
$$ z^{1-o(1)} \leq p_{1000} \leq \dots \leq p_0 \leq z^{1+o(1)}$$
such that $m = p_1 \dots p_{1000} m'$ with $m'$ $p_{1000}$-smooth.  In particular
\begin{equation}\label{m'-bound}
    m' \leq \frac{2x}{p_0^2 p_1 \dots p_{1000}}.
\end{equation}
\end{itemize}
Otherwise, the normalized interval is said to be \emph{non-typical}.
\end{definition}

We first handle all the non-typical intervals:

\begin{proposition}[Non-typical intervals]\label{prop-non-typical}  The union of non-typical normalized intervals has cardinality bounded by \eqref{bond}.
\end{proposition}

\begin{proof}
From \Cref{b-asym}(i) we see that any exceptional set of size
\begin{equation}\label{weak}
\ll \frac{x}{z^{2+\delta+o(1)}}
\end{equation}
for any fixed $\delta > 0$ is acceptable for this proposition. In particular, any exceptional set of size
\begin{equation}\label{strong}
\ll x^{1-c+o(1)}
\end{equation}
for a fixed $c>0$ is acceptable.

We first dispose of the case when $H \geq x^{0.14}$. Since
$$0.14 > \frac{2}{15} = 0.133\dots,$$
we can apply \Cref{prime-interval}(iii) to conclude that the set $E$ of $x' \in [0,2x]$ for which the interval $[x', x'+H/2]$ contains no prime has measure $O(x^{1-c+o(1)})$ for some fixed $c>0$.  On the other hand, if $n$ lies in $\{N+1,\dots,N+H\}$, then by \Cref{lem-basic} the interval $\{N+1,\dots,N+H\}$ cannot contain any primes, and hence the intersection of $E$ with $[n-2H, n+2H]$ will have measure $\gg H$.  By double counting (and a dyadic pigeonholing over the possible $H$), the set of $n$ covered by such intervals $\{N+1,\dots,N+H\}$ thus has cardinality at most $O(x^{1-c+o(1)})$, which is of the acceptable form \eqref{strong}. Thus we may assume that
\begin{equation}\label{H-upper}
    H < x^{0.14}.
\end{equation}

Next, we dispose of the case where $p_0 > x^{0.15}$.  For each $p_0$ there are only $O(\frac{x}{p_0^2})$ choices for $m$, and each $p_0$ and $m$ covers at most $O(x^{0.14})$ numbers by \eqref{H-upper}, so the total number of $n$ covered by this case is
$$ \ll \sum_{x^{0.15} \ll p_0 \ll x} \frac{x}{p_0^2} x^{0.14} \ll x^{1-0.01}$$
which is of the acceptable form \eqref{strong}.  Thus we may assume henceforth that
\begin{equation}\label{p0-upper}
    p_0 \leq x^{0.15}.
\end{equation}

We can also dispose of an easy case when $p_0 \leq z^{1/2-\eps}$ for some $\eps>0$.  In this case, all the elements of $\{N+1,\dots,N+H\}$ are $z^{1-\eps}$-smooth, so by \Cref{smoot}(i), the union of all these intervals has cardinality of the acceptable form \eqref{weak}.
Thus we may assume that $p_0 \geq z^{1/2-\eps}$.  By letting $\eps$ decay to zero sufficiently slowly, we may in fact assume that
\begin{equation}\label{p0-lower}
    p_0 \geq z^{\frac{1}{2}-o(1)}.
\end{equation}

We now consider the contribution of those intervals for which (i) fails, that is to say that
\begin{equation}\label{Hlarge}
    H \geq \log^{20} x.
\end{equation}
Let us temporarily fix $p_0$ and $H$.  For each prime $p_0 < p < 2p_0$, none of the elements of the interval \eqref{first} is divisible by $p$, thanks to \Cref{lem-basic}.  This excludes $H$ residue classes from $m$ for each prime $p_0 < p < 2p_0$.
Applying \Cref{large-sieve}, the total number of $m$ for this given choice of $H$ and $p_0$ is thus
\begin{equation}\label{ko}
 \ll \frac{x/p_0^2}{\left( \frac{1}{k} \sum_{p_0 < p < 2p_0}^{[-k]}\frac{H}{p} \right)^k}
\end{equation}
where $k = k(p_0)$ is the largest natural number for which
$$ (2p_0)^k \leq \left(\frac{x}{p_0^2}\right)^{1/2}.$$
From \eqref{p0-upper}, \eqref{p0-lower} we have\footnote{The lower bound on $k$ arises since $0.15 < \frac{1}{6}$.  This is why it is important to use the recent zero-density estimates of \cite{guth-maynard}; more classical zero-density estimates, such as that of Huxley \cite{huxley}, would not quite have sufficed.}
\begin{equation}\label{krange}
2 \leq k \ll \log^{1/2+o(1)} x.
\end{equation}
In particular, the removal of the $k$ largest elements in the sum in \eqref{ko} is of negligible impact, thanks to \eqref{p0-lower}.
Using the prime number theorem, we can bound \eqref{ko} by
$$ \ll \frac{x}{p_0^2} O\left( \frac{k \log p_0}{H} \right)^k.$$
As each interval is of length $H$, and $p_0 > H$ by \Cref{lem-basic}, the total size of this contribution is
$$ \ll \sum_{z^{1/2-o(1)} \leq p_0 \leq x^{0.15}} \sum_{\log^{20} x \leq H < p_0} H \frac{x}{p_0^2} O\left( \frac{k \log p_0}{H} \right)^k$$
where $H$ is understood to be restricted to powers of two.  By construction of $k$ we have $(2p_0)^k \leq x^{1/2}$, and hence $k \log p_0 \leq \frac{1}{2} \log x$.  From \eqref{Hlarge} we conclude that
$$ O\left( \frac{k \log p_0}{H} \right) \ll \frac{\log x}{H} \ll H^{-0.95} \leq H^{-0.9}$$
(for $x$ large enough), and hence, using the lower bound $k \geq 2$ from \eqref{krange},
$$ H \times O\left( \frac{k \log p_0}{H} \right)^k \leq H^{1-0.9k} \leq H^{-0.4k}.$$
Also, by construction of $k$ one has
$$ (2p_0)^{k+1} > (x/p_0^2)^{1/2}$$
and thus
$$ 2k+4 \geq (1-o(1)) \frac{\log x}{\log p_0}.$$
Since $2k+4 \leq 4k$ (again using $k \geq 2$), we conclude that
$$ k \geq \left(\frac{1}{4}-o(1)\right) \frac{\log x}{\log p_0}.$$
By \eqref{Hlarge} we then have
$$ H^{-0.4k} \leq \exp( - 8 k \log_2 x ) \leq \exp\left( - (2-o(1)) \frac{\log x \log_2 x}{\log p_0} \right),$$
and so the net contribution here is bounded by
$$ \ll (\log x) \sum_{z^{1/2-o(1)} \leq p_0 \leq x^{0.15}} \frac{x}{p_0}
\exp\left( - (2-o(1)) \frac{\log x \log_2 x}{\log p_0} - \log p_0 \right)$$
which by the arithmetic mean-geometric mean inequality can be bounded by
$$ \ll (\log x) \sum_{z^{1/2-o(1)} \leq p_0 \leq x^{0.15}} \frac{x}{p_0}
\exp( - (2\sqrt{2}-o(1)) \log^{1/2} x \log^{1/2}_2 x ) = \frac{x}{z^{4+o(1)}}$$
which is of the acceptable form \eqref{weak}.\footnote{There is considerable room in this argument: any bound of the form $\exp( - c \frac{\log x \log_2 x}{\log p_0})$ with $c > \frac{1}{2}$ would have sufficed here.}
Thus we may assume henceforth that (i) holds.

Suppose that $\{N+1,\dots,N+H\}$ contains a multiple of a square $d^2$ with $d \geq z^3$.  For a given $d$, there are $O(\frac{x}{d^2})$ such multiples, each of which impacts at most $H = O(\log^{20} x)$ values of $n$ by \eqref{hup}.  Thus the total contribution of this case is
$$ \ll \sum_{d \geq z^3} \frac{x}{d^2} \log^{20} x \ll \frac{x}{z^{3-o(1)}}$$
which is of the acceptable form \eqref{weak}. Thus we may now assume that (ii) holds. Among other things, this implies that $p_0 < z^3$.

Suppose that $p_0 \geq z^{1+\eps}$ for some fixed $\eps>0$.  Writing $p_0 = z^{\alpha}$ for some $1+\eps \leq \alpha \leq 3$, we have from \Cref{smoot}(i) that the number of $p_0$-smooth numbers up to $x/p_0^2$ is
$$ \ll \frac{x}{p_0^2 z^{1/\alpha+o(1)}} = \frac{x}{p_0 z^{\alpha+1/\alpha+o(1)}} \leq \frac{x}{p_0 z^{2+\delta+o(1)}}$$
for some $\delta>0$ depending on $\eps$.  Applying \eqref{hup}, we conclude that the union of bad intervals covered by this case has cardinality at most
$$ \ll \sum_{z^{1+\eps+o(1)} \leq p_0 \leq z^3} \frac{x}{p_0 z^{2+\delta+o(1)}} \log^{20} x \ll \frac{x}{z^{2+\delta+o(1)}}$$
which is of the acceptable form \eqref{weak}.  Thus we may now assume that $p_0 < z^{1+\eps}$.  By sending $\eps$ to zero sufficiently slowly, we may now assume that
\begin{equation}\label{p0-upper-2}
    p_0 \leq z^{1+o(1)}.
\end{equation}

Let $\eps>0$ be fixed.  Suppose that $m$ has fewer than $1000$ prime factors exceeding $z^{1-\eps}$.
Then we can write $m = p_1 \dots p_j m'$ where $j < 1000$, $p_j > z^{1-\eps}$, and $m'$ is $z^{1-\eps}$-smooth and of size at most $\frac{x}{p_0^2 p_1 \dots p_j}$.  For fixed $p_0,\dots,p_j$, we can apply \Cref{smoot}(i) (with $\alpha = 1-\eps$) to bound the number of $m'$ by
$$ \ll \frac{1}{z^{\frac{1}{1-\eps}-o(1)}} \frac{x}{p_0^2 p_1 \dots p_j},$$
so the total number of $n$ covered by this case is at most
$$ \ll \frac{x \log^{20} x}{z^{\frac{1}{1-\eps}-o(1)}}
\sum_{1 \leq j < 1000} \sum_{z^{1+o(1)} > p_0 \geq \dots \geq p_j > z^{1-\eps}} \frac{1}{p_0^2 p_1 \dots p_j}.$$
The inner sum can be evaluated to be $O( 1 / z^{1-\eps})$ (by summing $p_0, p_1, \dots, p_j$ in turn).  Since $\frac{1}{1-\eps} + 1-\eps > 2+\delta$ for some $\delta>0$, this is of the acceptable form \eqref{weak}.  Thus we may assume that $m$ has at least $1000$ prime factors exceeding $z^{1-\eps}$, but bounded by $z^{1+o(1)}$ thanks to \eqref{p0-upper-2}.  By sending $\eps$ sufficiently slowly to zero, we obtain (iii).
\end{proof}

\subsection{Reduction to a probability estimate}

It remains to show that the union of typical normalized intervals has cardinality bounded by \eqref{bond}.

As $H$ is a power of two that is at most $\log^{20} x$, there are only $O(\log_2 x) = O(\log^{o(1)} z)$ choices for $H$.  Thus we may take $H$ to be fixed without impacting \eqref{bond}.  As each interval contributes $H$ numbers to the final count, it suffices to show that the number of tuples
\begin{equation}\label{tuple}
 (m', p_0, p_1, \dots, p_{1000})
\end{equation}
such that the interval
\begin{equation}\label{pong}
 \{ p_0^2 p_1 \dots p_{1000} m', \dots, p_0^2 p_1 \dots p_{1000} m' + H - 1 \}
\end{equation}
is a typical interval with the indicated parameters $m', p_0,\dots,p_{1000}$, is at most
\begin{equation}\label{tok}
 \ll \frac{\#({\mathcal B}^1 \cap [1,x])}{H \log^{2-o(1)} z}.
 \end{equation}

We can place each $p_j$ for $j=0,\dots,1000$ in a dyadic interval $[P_j, 2P_j)$ with $P_j$ a power of two and
$$ z^{1-o(1)} \leq P_{1000} \leq \dots \leq P_{0} \leq z^{1+o(1)},$$
so now by \eqref{m'-bound} we have
\begin{equation}\label{m'-bound-alt}
    m' \leq \frac{2x}{P_0^2 P_1 \dots P_{1000}}.
\end{equation}
The key estimate is the following probability bound.

\begin{proposition}[Probability bound]\label{prop-key}  Let $P_0,\dots,P_{1000}, m'$ be as above.
Let  $\mathbf{p}_j$ for $j=0,\dots,1000$ be a prime drawn uniformly from $[P_j, 2P_j)$, independently in $j$.  Let $E$ denote the event that the interval
$$ \{ \mathbf{p}_0^2 \mathbf{p}_1 \dots \mathbf{p}_{1000} m', \dots, \mathbf{p}_0^2 \mathbf{p}_1 \dots \mathbf{p}_{1000} m' + H - 1 \} $$
is a typical normalized interval.  Then
$$ \P( E ) \ll \frac{1}{H\log^{2-o(1)} z}$$
\end{proposition}

Indeed, assume that this proposition held.  Then by the prime number theorem and the $2P_{1000}$-smoothness of $m'$, the number of tuples \eqref{tuple} for which \eqref{pong} is a typical normalized admissible interval is
\begin{equation}\label{sum}
\ll \sum_{z^{1-o(1)} \leq P_{1000} \leq \dots \leq P_0 \leq z^{1+o(1)}} \Psi\left(\frac{2x }{ P_0^2 P_1 \dots P_{1000}}, 2P_{1000}\right) \frac{\prod_{j=0}^{1000} \frac{P_j}{\log P_j}}{H\log^{2-o(1)} z}.
\end{equation}
By the prime number theorem again, we can bound the summand by
$$ \ll  \sum_{p_0 \in [4P_0,8P_0)} \sum_{p_1 \in [2P_1,4P_1)} \dots \sum_{p_{1000} \in [2P_{1000},4P_{1000})} \sum_{\substack{m' \leq 2x/P_0^2 P_1 \dots P_{1000}\\ m'\ 2P_{1000}\text{-smooth}}} \frac{1}{H\log^{2-o(1)} z}.$$
Note that each term here creates an element $m = p_0^2 p_1 \dots p_{1000} m'$ of ${\mathcal B}^1 \cap [1,O(x)]$, and each such element $m$ is represented at most $O(1)$ times (note from construction that $p_0,\dots,p_{1000}$ exceed any prime factor of $m'$, so that $p_0^2 p_1 \dots p_{1000}$ is the product of the $1002$ largest prime factors of $m$, counting multiplicity).  Thus we can bound \eqref{sum} by
$$ \ll  \frac{ \# ({\mathcal B}^1 \cap [1,O(x)])}{H\log^{2-o(1)} z},$$
which gives \eqref{tok} thanks to \Cref{b-asym}(ii).

\subsection{Setting up an anti-sieve}

It remains to prove \Cref{prop-key}. Suppose the event $E$ holds.  For $1 \leq l \leq H-1$, the squarefree part of $\mathbf{p}_{0}^2 \mathbf{p}_{1} \dots \mathbf{p}_{1000} m' + l$ is bounded below by $\gg x/z^6$ and also $z^{1+o(1)}$-smooth, by \Cref{def-typical} and \Cref{lem-basic}.  Applying the fundamental theorem of arithmetic and taking logarithms, we conclude that
$$ \sum_{p \leq z^{1+o(1)}} \mathbf{X}_{p,l} \log p \gg \log \frac{x}{z^6} = \log^{2-o(1)} z,$$
where $\mathbf{X}_{p,l}$ is the indicator function of the event that $p$ divides $\mathbf{p}_{0}^2 \mathbf{p}_{1} \dots \mathbf{p}_{1000} m' + l$.  Summing in $l$, we thus have
$$ \sum_{l=1}^{H-1} \sum_{p \leq z^{1+o(1)}} \mathbf{X}_{p,l} \log p \gg H \log^{2-o(1)} z.$$
It therefore suffices to establish the bound
$$ \P \left( \sum_{l=1}^{H-1} \sum_{p \leq z^{1+o(1)}} \mathbf{X}_{p,l} \log p \gg H \log^{2-o(1)} z \right) \ll \frac{1}{H \log^{2-o(1)} z}.$$
Intuitively, we expect the $\mathbf{X}_{p,l}$ to behave like independent Bernoulli random variables with mean $1/\phi(p)$, which would imply that the random variable $\sum_{l=1}^{H-1} \sum_{p \leq z^{1+o(1)}} \mathbf{X}_{p,l} \log p$ has mean around $H \log z$; thus we have created an ``anti-sieve'' situation in which an unexpectedly large number of small primes $p \leq z^{1+o(1)}$ divide the $\mathbf{p}_{0}^2 \mathbf{p}_{1} \dots \mathbf{p}_{1000} m' + l$.  To make this intuition rigorous, we will use the moment method, taking advantage of the character sum estimates in \Cref{primesum}, together with the randomness of the $1000$ prime factors $\mathbf{p}_1,\dots,\mathbf{p}_{1000}$, to generate enough ``mixing'' to obtain good control on moments.

It will be convenient to treat different classes of primes $p$ separately.  More precisely, we shall establish the bounds
\begin{align}
     \P \left( \sum_{l=1}^{H-1} \sum_{p | l} \mathbf{X}_{p,l} \log p \gg H \log^{2-o(1)} z \right) &= 0\label{excep-prime}\\
     \P \left( \sum_{l=1}^{H-1} \sum_{p \leq z^{1/100}: p \nmid l} \mathbf{X}_{p,l} \log p \gg H \log^{2-o(1)} z \right) &\ll \frac{1}{\log^{50-o(1)} z}\label{small-prime}\\
     \P \left( \sum_{l=1}^{H-1} \sum_{z^{1/100} < p \leq z^{1+o(1)}} \mathbf{X}_{p,l} \log p \gg H \log^{2-o(1)} z \right) &\ll \frac{1}{H \log^{2-o(1)} z}.\label{large-prime}
\end{align}
which give the claim by \eqref{hup} and the triangle inequality.

The estimate \eqref{excep-prime} is easy to establish.  This contribution can only occur when $p \ll H \ll \log^{20} x$, and for each such $p$ only $O(H/p)$ values of $l$ contribute. Thus
$$\sum_{l=1}^{H-1} \sum_{p | l} \mathbf{X}_{p,l} \log p \ll  \sum_{p \ll \log^{20} x} \frac{H}{p} \log(\log^{20} x) \ll H \log_2^2 x \asymp H \log_2^2 z$$
giving the claim \eqref{excep-prime}.

\subsection{Small primes}

Now we handle the contribution \eqref{small-prime} of the small primes.  Here we use the moment method with a high moment in order to overcome a certain loss arising from a use of \eqref{hup}.  By Markov's inequality, it suffices to show that
\begin{equation}\label{nsum}
 \E \left( \sum_{l=1}^{H-1} \sum_{p \leq z^{1/100}: p \nmid l} \mathbf{X}_{p,l} \log p \right)^{50} \ll H^{50} \log^{50} z.
\end{equation}

The left-hand side of \eqref{nsum} can be expanded as
$$ \sum_{1 \leq l_1,\dots,l_{50} \leq H-1} \sum_{\substack{p'_1,\dots,p'_{50} \leq z^{1/100} \\ p'_i \nmid l_i}} (\log p'_1) \dots (\log p'_{50}) \E \mathbf{X}_{p'_1,l_1} \dots \mathbf{X}_{p'_{50},l_{50}}.$$
The event $\mathbf{X}_{p'_1,l_1} \dots \mathbf{X}_{p'_{50},l_{50}}=1$ is either impossible to satisfy, or constrains $\mathbf{p}_{0}^2 \mathbf{p}_{1} \dots \mathbf{p}_{1000}$ to a single primitive residue class $a$ modulo $[p'_1,\dots,p'_{50}]$ by the Chinese remainder theorem (we allow $a$ to depend on $m'$, the $p'_i$, and the $l_i$).  In the latter case, we can expand  into Dirichlet characters and use the triangle inequality to obtain
\begin{align*}
|\E \mathbf{X}_{p'_1,l_1} \dots \mathbf{X}_{p'_{50},l_{50}}| &=
\frac{1}{\phi([p'_1,\dots,p'_{50}])} \left|\sum_{\chi\ ([p'_1,\dots,p'_{50}])} \E \chi(\mathbf{p}_0^2 \mathbf{p}_1 \dots \mathbf{p}_{1000}) \overline{\chi}(a)\right| \\
&\ll \frac{1}{[p'_1,\dots,p'_{50}]} \sum_{\chi\ ([p'_1,\dots,p'_{50}])} |\E \chi(\mathbf{p}_0^2 \mathbf{p}_1 \dots \mathbf{p}_{1000})|,
\end{align*}
where we have used the fact that $[p'_1,\dots,p'_{50}]$ has at most $50$ prime factors to conclude that $\phi([p'_1,\dots,p'_{50}]) \gg [p'_1,\dots,p'_{50}]$.
This bound trivially holds in the former case also.
By independence we can factor
$$\E \chi(\mathbf{p}_0^2 \mathbf{p}_1 \dots \mathbf{p}_{1000}) = (\E \chi(\mathbf{p}_0^2)) \prod_{i=1}^{1000} (\E \chi(\mathbf{p}_i)).$$
We discard the $\E \chi(\mathbf{p}_0^2)$ term and use \eqref{sz-def} and the arithmetic-geometric mean inequality to bound
\begin{equation}\label{amgm}
|\E \chi(\mathbf{p}_0^2 \mathbf{p}_1 \dots \mathbf{p}_{1000})|  \ll \sum_{i=1}^{1000} |s_{P_i}(\chi)|^{1000}.
\end{equation}
Thus by the pigeonhole principle we can bound the left-hand side of \eqref{nsum} by
$$ \ll H^{50} \sum_{p'_1,\dots,p'_{50} \leq z^{1/100}} \frac{(\log p'_1) \dots (\log p'_{50})}{[p'_1,\dots,p'_{50}]} \sum_{\chi\ ([p'_1,\dots,p'_{50}])} |s_{P_i}(\chi)|^{1000}$$
for some $1 \leq i \leq 1000$, which we now fix.

The conductor of $\chi$ is some product $p''_1 \dots p''_j$ of $j$ distinct primes $p''_1 < \dots < p''_j$ that appear amongst the $p'_1,\dots,p'_{50}$ for some $0 \leq j \leq 50$.  If one fixes these $p''_1,\dots,p''_j$ and sums the quantity $\frac{(\log p'_1) \dots (\log p'_{50})}{[p'_1,\dots,p'_{50}]}$ over all $p'_1,\dots,p'_{50}$ that contain the $p''_1,\dots,p''_j$, one can see from many applications of Mertens' theorem that this sum is
$$ \ll \frac{\log p''_1 \dots \log p''_j}{p''_1 \dots p''_j} \log^{50-j} z.$$
Thus the left-hand side of \eqref{nsum} is bounded by
$$ \ll \sum_{j=0}^{50} H^{50} \log^{50-j} z \sum_{p''_1 < \dots < p''_j \leq z^{1/100}} \frac{\log p''_1 \dots \log p''_j}{p''_1 \dots p''_j} \sum_{\chi\ (p''_1 \dots p''_j)}^* |s_{P_i}(\chi)|^{1000},$$
where the asterisk in the sum indicates that $\chi$ is required to be primitive\footnote{Note from \eqref{sz-def} that the primes implicitly being summed in $s_{P_i}(\chi)$ have magnitude $\asymp P_i = z^{1+o(1)}$ and are thus larger than the primes $p''_j = O(z^{1/100})$ dividing the conductor.  It is therefore safe to replace a non-primitive character $\chi$ here by its primitive counterpart.}.  Discarding the bounded factor
$\frac{\log p''_1 \dots \log p''_j}{p''_1 \dots p''_j}$, we now see that to show \eqref{nsum}, it suffices to show that
$$\sum_{p''_1 < \dots < p''_j \leq z^{1/100}} \sum_{\chi\ (p''_1 \dots p''_j)}^* |s_{P_i}(\chi)|^{1000} \ll 1$$
for all $0 \leq j \leq 50$ and all $p''_1 < \dots < p''_j \leq z^{1/100}$.  Each unexceptional character $\chi$ contributes at most $(P_i^{-0.008})^{1000} = z^{-8+o(1)}$ to the left-hand side, while the number of characters $\chi$ is at most $(z^{j/100})^2 = O(z)$, so the contribution of the unexceptional characters is $O(z^{-7+o(1)})$, which is acceptable.  Meanwhile, from \Cref{primesum} (and discarding all but two of the factors of $|s_{P_i}(\chi)|$) we see that the contribution of the exceptional characters is also acceptable.  The claim follows.

\subsection{Large primes}

It remains to handle the contribution \eqref{large-prime} of the large primes $z^{1/100} < p \leq z^{1+o(1)}$.  Here $\log p \asymp \log z$, so we may rewrite the bound as
$$ \P \left( \sum_{l=1}^{H-1} \sum_{z^{1/100} < p \leq z^{1+o(1)}} \mathbf{X}_{p,l} \gg H \log^{1-o(1)} z \right) \ll \frac{1}{H \log^{2-o(1)} z}.
$$
By Chebyshev's inequality, it will suffice to show the mean and variance bounds
\begin{equation}\label{mean}
    \E \sum_{l=1}^{H-1} \sum_{z^{1/100} < p \leq z^{1+o(1)}} \mathbf{X}_{p,l} \ll H
\end{equation}
\begin{equation}\label{var}
    \Var \left( \sum_{l=1}^{H-1} \sum_{z^{1/100} < p \leq z^{1+o(1)}} \mathbf{X}_{p,l} \right) \ll H.
\end{equation}

The mean estimate \eqref{mean} will follow from

\begin{proposition}[First moment of $\mathbf{X}_{p,l}$]\label{prop-first} Let $1 \leq l \leq H-1$.
\begin{itemize}
    \item[(i)] For $z^{1/100} \ll p \ll z^{1+o(1)}$, one has
\begin{equation}\label{xpl-crude}
\E \mathbf{X}_{p,l} \ll \frac{\log^2 z}{p}.
\end{equation}
    \item[(ii)] For $z^{1/100} \ll P \ll z^{1+o(1)}$, one has the improved bound
\begin{equation}\label{zip-2}
    \E \mathbf{X}_{p,l} = \frac{1}{\phi(p)} + O\left(\frac{1}{P^{1.001}}\right) \ll \frac{1}{p}
\end{equation}
for all but $O(P^{0.02+o(1)})$ primes $p$ in $[P,2P)$.
\end{itemize}
\end{proposition}

Indeed, from \Cref{prop-first} and the prime number theorem one has
$$ \sum_{P \leq p < 2P} \E \mathbf{X}_{p,l} \ll \frac{1}{\log P} + \frac{P^{0.02+o(1)} \log^2 z}{P} \ll \frac{1}{\log z}$$
for any $z^{1/100} \ll P \ll z^{1+o(1)}$, and the claim \eqref{mean} then follows from dyadic summation and the triangle inequality.

\begin{proof}  We begin with (i). Freeze all random variables except for $\mathbf{p}_1$ and $\mathbf{p}_2$.  The event $\mathbf{X}_{p,l}=1$ forces $\mathbf{p}_1\mathbf{p}_2$ to lie in some residue class modulo $p$; on the other hand, by the prime number theorem $\mathbf{p}_1\mathbf{p}_2$ takes $\asymp \frac{P_1 P_2}{\log^2 z}$ values in $[P_1P_2, 4P_1P_2)$, with each value attained with probability $\asymp \frac{\log^2 z}{P_1 P_2}$.  Since $p, P_1, P_2 = z^{1+o(1)}$, this residue class intersects $[P_1P_2, 4P_1P_2)$ in $O( \frac{P_1 P_2}{p} )$ values, giving the claim \eqref{xpl-crude}.

Now we establish (ii).  The condition $\mathbf{X}_{p,l}=1$ is equivalent to a congruence condition
$$ \mathbf{p}_0^2 \mathbf{p}_1 \dots \mathbf{p}_{1000} = a\ (p)$$
for some residue class $a$ depending on $p, m', l$.  We can exclude the cases where $p$ divides $m'$, as this only occurs $O(\log x)$ times; and as $l \leq H \ll \log^{20} x$ and $p \geq z^{1/100}$, $p$ also does not divide $l$.  Thus $a$ is coprime to $p$.  We can then perform a character expansion
$$ 1_{\mathbf{p}_0^2\mathbf{p}_1 \dots \mathbf{p}_{1000} = a\ (p)} =  \frac{1}{\phi(p)} \left( 1 + \sum_{\chi\ (p)}^* \chi(\mathbf{p}_0^2\mathbf{p}_1 \dots \mathbf{p}_{1000}) \overline{\chi}(a) \right) + O\left(\frac{\sum_{i=1}^{1000} 1_{p = \mathbf{p}_i}}{P}\right) $$
where the error term comes from the vanishing of the principal character $\chi_0(\mathbf{p}_0^2\mathbf{p}_1 \dots \mathbf{p}_{1000})$ when one of the $\mathbf{p}_i$ equals $p$.  On taking expectations and using \eqref{amgm} and the triangle inequality, we will establish \eqref{zip-2} whenever
\begin{equation}\label{ta}
\sum_{\chi\ (p)}^* |s_{P_i}(\chi)|^{1000} \ll \frac{1}{P^{0.001}}
\end{equation}
for all $i=1,\dots,1000$.
But by \Cref{primesum} (with $q_1=1$), we have for each $i$ that for all but $O(P^{0.02+o(1)})$ choices of $p$, one has $|s_{P_i}(\chi)| \ll P^{-0.01}$ for all primitive characters $\chi$ of conductor $p$; and the claim follows.  This concludes the proof of \eqref{zip-2}.
\end{proof}

We now prove \eqref{var}.  The variable $X_{p,l}$ vanishes if $p$ divides $m'$, so we remove those variables from consideration. The left-hand side expands as
$$
\sum_{1 \leq l,l' \leq H-1} \sum_{\substack{z^{1/100}\leq p,p' \leq z^{1+o(1)} \\ p,p' \nmid m'}} \Cov \left( \mathbf{X}_{p,l}, \mathbf{X}_{p',l'} \right).
$$
The diagonal cases $(p,l) = (p',l')$ are acceptable thanks to \eqref{mean} and Mertens' theorem, since the variance of a boolean random variable is bounded by the mean.  If $p=p'$ and $l \neq l'$, then $\mathbf{X}_{p,l} \mathbf{X}_{p',l'}$ vanishes since $p \geq z^{1/100}$ does not divide $|l-l'| \leq \log^{20} x$.  Hence the covariances of those terms are negative and can be discarded. Thus by symmetry we may assume that $p < p'$. It now suffices to show that
\begin{equation}\label{pout}
\sum_{1 \leq l,l' \leq H-1} \sum_{\substack{z^{1/100} \leq p < p' \leq z^{1+o(1)} \\ p, p' \nmid m'}} \Cov \left( \mathbf{X}_{p,l}, \mathbf{X}_{p',l'} \right) \ll H.
\end{equation}

We have the following analogue of \Cref{prop-first}:

\begin{proposition}[Second moment of $\mathbf{X}_{p,l}$]\label{prop-second} Let $1 \leq l,l' \leq H-1$.
\begin{itemize}
    \item[(i)] For $z^{1/100} \ll p < p' \ll z^{1+o(1)}$, one has
\begin{equation}\label{cortriv}
 \E \mathbf{X}_{p,l} \mathbf{X}_{p',l'} \ll \frac{\log^3 z}{p p'}.
\end{equation}
    \item[(ii)] For $z^{1/100} \ll P \leq P' \ll z^{1+o(1)}$, one has the improved bound
\begin{equation}\label{lop}
    \E \mathbf{X}_{p,l} \mathbf{X}_{p',l'} = \frac{1}{\phi(pp')} + O\left(\frac{1}{P^{1.001} P'}\right)
\end{equation}
for all but $O(P^{0.02+o(1)} P')$ of the pairs $(p,p')$ of primes in $[P,2P) \times [P', 2P')$ with $p < p'$.
\end{itemize}
\end{proposition}

Assuming this proposition, then for any $z^{1/100} \ll P \leq P' \ll z^{1+o(1)}$ we see from \eqref{xpl-crude}, \eqref{cortriv} that
$$ \Cov\left( \mathbf{X}_{p,l}, \mathbf{X}_{p',l'}\right) \ll \frac{\log^3 z}{PP'}$$
for all pairs $(p,p') \in [P,2P) \times [P', 2P')$ with $p < p'$, while from \eqref{zip-2}, \eqref{lop} we have the improved bound
$$ \Cov\left( \mathbf{X}_{p,l}, \mathbf{X}_{p',l'}\right) \ll \frac{1}{P^{1.001} P'} $$
for all but $O(P^{0.02+o(1)} P')$ of the pairs.  Summing,
we conclude that
$$
\sum_{1 \leq l,l' \leq H-1} \sum_{\substack{P \leq p < 2P \\ p \nmid m'}} \sum_{\substack{P' \leq p' < 2P' \\ p' \nmid m', p' > p}} \Cov \left( \mathbf{X}_{p,l}, \mathbf{X}_{p',l'} \right) \ll \frac{H^2 \log^3 z}{P^{0.001}}  $$
and on summing over dyadic $z^{1/100} \leq P \leq P' \leq z^{1+o(1)}$ and using \eqref{hup} we obtain \eqref{pout} as desired (with room to spare).

\begin{proof}
    We begin with (i). After freezing all the $\mathbf{p}_i$ except for $\mathbf{p}_1, \mathbf{p}_2, \mathbf{p}_3$, the event $\mathbf{X}_{p,l} \mathbf{X}_{p',l'}=1$ is either empty, or equivalent by the Chinese remainder theorem to a congruence condition
$$ \mathbf{p}_1 \mathbf{p}_2 \mathbf{p}_3 = a\ (p p')$$
for some primitive residue class $a$ depending on $p, p', m', l, l'$ and the frozen primes.  This restricts $\mathbf{p}_1 \mathbf{p}_2 \mathbf{p}_3$ to an arithmetic progression in $[P_1P_2P_3, 8P_1P_2P_3]$ of size $O(P_1 P_2 P_3 / pp')$, but by the prime number theorem this product takes $\frac{\asymp P_1 P_2 P_3 }{ \log^3 z}$ values, each of which is attained with probability $\asymp \frac{\log^3 z}{P_1 P_2 P_3}$, giving the claim \eqref{cortriv}.

Now we establish \eqref{lop} for all but $O(P^{0.02+o(1)} P' + P (P')^{0.02+o(1)})$ of  the pairs $(p,p')$.  The condition $\mathbf{X}_{p,l} \mathbf{X}_{p',l'}=1$ is equivalent to a congruence condition
$$ \mathbf{p}_{0}^2 \mathbf{p}_{1} \dots \mathbf{p}_{1000} = a\ (pp')$$
for some primitive residue class $a$.  By expansion into characters as before, we conclude that
$$ \left|\E \mathbf{X}_{p,l} \mathbf{X}_{p',l'} - \frac{1}{\phi(pp')}\right| \leq \frac{1}{\phi(pp')} \sum_{\substack{\chi\ (pp') \\ \chi \neq \chi_0}} |\E \chi(\mathbf{p}_{0}^2 \mathbf{p}_{1} \dots \mathbf{p}_{1000})| + O\left( \frac{1}{P^{1.001} P'}\right).$$
Thus, to obtain \eqref{lop}, it will suffice by \eqref{amgm} to show that
\begin{equation}\label{stomp}
 \sum_{\substack{\chi\ (pp') \\ \chi \neq \chi_0}} |s_{P_i}(\chi)|^{1000} \ll P^{-0.001}
\end{equation}
for each $i=1,\dots,1000$.
By \eqref{ta}, the contribution of those $\chi$ that have conductor $p$ will be acceptable after removing $O(P^{0.02+o(1)})$ choices of $p$.  Similarly, the contribution of those $\chi$ that have conductor $p'$ will be acceptable after removing $O((P')^{0.02+o(1)})$ choices of $p'$.  Thus we may restrict attention to the primitive characters $\chi$ of conductor $pp'$. If we fix $p$, then from \Cref{primesum} (with $q_1=p$) we see that for all but $O((P')^{0.02+o(1)})$ choices of $p'$, one has $|s_{P_i}(\chi)| \ll (P')^{-0.01}$ for all $i=1,\dots,1000$ and all primitive characters $\chi$ of conductor $pp'$, which would imply \eqref{stomp} since
$$ (PP') ((P')^{-0.01})^{1000} \ll P^{-0.001}.$$
There are only $O(P (P')^{0.02+o(1)})$ pairs $(p,p')$ not covered by this argument, giving \eqref{lop} with the desired bound on the exceptional pairs (noting that $P \leq P'$).
\end{proof}

This completes the proof of \Cref{prop-key}, and hence \Cref{bad-thm}.

\bibliographystyle{amsplain}

\end{document}